\documentclass[preprint,review,10pt]{elsarticle}
\usepackage{bbm}
\usepackage{amsmath,amssymb,amsthm}
\usepackage{bm}
\usepackage{color}
\usepackage{xcolor}
\usepackage{graphicx}
\usepackage{enumerate}
\usepackage{ifpdf}
\usepackage{algorithmic}
\usepackage{float}
\usepackage{CJK}
\usepackage{listings}
\usepackage{bbm}
\usepackage{mathrsfs}
\usepackage{amsmath}
\usepackage{txfonts}
\usepackage{epstopdf}
\usepackage{cases}
\usepackage{subfigure}
\usepackage{enumerate}
\usepackage{lipsum}
\makeatletter
\def\ps@pprintTitle{%
 \let\@oddhead\@empty
 \let\@evenhead\@empty
 \def\@oddfoot{}%
 \let\@evenfoot\@oddfoot}
\makeatother
 \usepackage{graphicx}
\usepackage{amssymb}
\usepackage{amsthm}

\makeatletter\@addtoreset{equation}{section} \makeatother
\newtheorem{theorem}{Theorem}[section]
\newtheorem{lemma}{Lemma}[section]
\newtheorem{proposition}{Proposition}[section]

\newtheorem{remark}{Remark}[section]

\newcommand{\R}{\mathbb{R}}

\begin{document}
\begin{frontmatter}
\title{Structure-preserving quasi-interpolation for \\ vector-valued function with multiple physical constraints}
\tnotetext[label1]{This work is supported by  Youth Project (Class A) of Anhui Provincial Natural Science Foundation (No. 2508085J009) and  Natural Science Foundation of China (No. 12271002).}
\author{Wenwu Gao}
 \ead{wenwugao528@163.com}
  \address{School of Big Data and Statistics,  Anhui University, Hefei, P. R. China}
 \begin{abstract}
We develop a unified theory for constructing vectorial quasi-interpolation that \textbf{exactly} preserves the intrinsic  physical structures of the approximand. These physical structures are characterized by a system of linear differential operators with constant coefficients, which contains the classical divergence-free and curl-free constraints as special cases. The central ingredient of our approach is a family of matrix-valued kernels, built from an orthogonal projector in the frequency domain, whose columns \textbf{analytically} satisfy the prescribed constraints. The resulting quasi-interpolation is a weighted average of translates of this kernel, the weights being the sampled function values; consequently, no constrained optimization problem has to be solved. Simultaneous error estimates for the approximand and its derivatives are derived within a bias-variance framework. To apply the proposed quasi-interpolation to the numerical decomposition of vector fields, we first present a generalized Helmholtz--Hodge decomposition that splits an arbitrary smooth vector field into two components endowed with distinct physical structures. For each component, we further devise structure-preserving quasi-interpolation schemes along with the corresponding error estimates. Explicit matrix-valued kernels are constructed for multiple typical physical  constraints, including steady-state acoustic constraints, combined divergence-curl constraints, coupled divergence-free constraints, and second-order Saint-Venant compatibility conditions, demonstrating  the broad applicability of the developed theory.  Numerical simulations validate the theoretical error bounds and the \textbf{exactness} of structure preservation. As illustrated in the numerical tests, this property enables accurate and efficient identification of sources and sinks within vector fields. Furthermore, a three-dimensional numerical experiment is presented to achieve numerical reconstruction of generalized Helmholtz--Hodge components for general smooth vector fields without prior constraint assumptions.
\end{abstract}


\end{frontmatter}


\section{\textit{Introduction}}
Approximating vector-valued functions is a fundamental problem in numerous scientific and engineering fields, including fluid mechanics, electromagnetics, meteorology, and optical flow analysis \cite{Benbourhim1,Kressner}. In these applications the vector field typically possesses intrinsic  physical structures, which are not merely incidental mathematical properties but essential manifestations of physical laws. Constructing numerical algorithms that \textbf{exactly} preserve such structures throughout the approximation process is therefore crucial for the reliability and physical consistency of numerical results, predictive models, and engineering simulations. Driven by this need, physics-informed approximation has emerged as a pivotal interdisciplinary research direction at the interface of approximation theory and computational physics \cite{Ryck}.

The divergence-free constraint (incompressible fluid flow) and the curl-free constraint (magnetic fields in classical electrodynamics) are the two most familiar instances of such physical structures. More importantly, the Helmholtz--Hodge decomposition \cite{Schwarz} shows that every smooth vector field splits into a divergence-free and a curl-free component. Accordingly, there has been a substantial body of work on approximating divergence-free and curl-free vector fields. Examples include divergence-free/curl-free interpolation (\cite{DoduandRabut}, \cite{Drake}, \cite{Fuselier, Fuselier1}, \cite{McNally}, \cite{Narcowich, Narcowich1}, \cite{VenelandBeatson}),  (moving) least squares approximation with divergence/curl constraints (\cite{Benbourhim0, Benbourhim1}, \cite{DoduandRabut0}, \cite{Freeden}, \cite{MirzaeiandMohammadi}), vector field decomposition (\cite{Deriaz}, \cite{Fuselier1, Fuselier2}, \cite{McNally}, \cite{pp}), and numerical solutions of certain differential equations (\cite{Wendland1}, \cite{Wendland3}, \cite{Wendland5}). Recently, by constructing a divergence-free matrix-valued kernel from thin plate splines (\cite{Rabut0, Rabut, Rabut1}), the authors of \cite{Gaoetal4} obtained a divergence-free approximation based on quasi-interpolation---a basic approximation tool that has been extensively studied and widely used, see \cite{Buhmann, Buhmann1, BuhmannandDyn, BuhmannBook}, \cite{Wendland7}, \cite{Gaoetal2, Gaoetal3,Gaoetal5, Gaoetal1}, \cite{Ortmann, Ortmann1}, \cite{CR, Rabut0}, \cite{ZMRS, GeneralizedWu} and the references therein. They subsequently extended the idea to divergence-free, curl-free, and harmonic quasi-interpolation, and applied it to the Helmholtz--Hodge decomposition \cite{Fisher}.

Other physical constraints have been studied as well. Amodei and Benbourhim \cite{Amodei} constructed a vector quasi-interpolant from a vector spline function, namely the unique minimizer of a quadratic functional involving both the divergence and the curl of the field. Atteia, Benbourhim, and Casanova \cite{Atteia} extended this idea to quasi-interpolation on spline elastic manifolds. Atteia \cite{Atteia1} introduced the notion of a $P(D)$ manifold, with $P(D)$ a first-order homogeneous differential operator, and constructed the corresponding quasi-interpolant. Benbourhim and Casanova \cite{Benbourhim} developed a technique based on generalized variation in tensor-product Sobolev spaces. Chen and Suter \cite{ChenandSuter} constructed div-curl vector quasi-interpolation on a bounded domain by means of boundary extension.

To the best of our knowledge, however, no theory is available for constructing vectorial quasi-interpolation that \textbf{exactly} preserves physical constraints of a more general form, and in particular several constraints simultaneously. The present paper aims to fill this gap.

We first generalize the conventional divergence-free constraint to a broader family of dot-product constraints  formulated via a system of $n$ linearly independent, constant-coefficient homogeneous differential equations: \begin{equation}\label{physicalinvariants}
P_l(-D)\cdot \bold u_m=0, \quad \text{for all}\quad 1\leq l\leq n.
 \end{equation}
Here $d$, $m$, $n$ are positive integers and $\alpha=(\alpha_1,\alpha_2,\cdots,\alpha_d)$ is a multi-index of length $|\alpha|=\sum_{j=1}^d\alpha_j$. We write $D^{\alpha}=\prod_{j=1}^d\partial_{x_j}^{\alpha_j}$ for the differential operator of order $|\alpha|$ with $\partial_{x_j}=\frac{\partial}{\partial x_j}$, and $P_l(D)=\sum_{|\alpha|\leq r_l}\bold{\gamma}_{l}^{\alpha}D^{\alpha}$ for an $m\times 1$ column vector differential operator of order $r_l$ with constant vector coefficients $\bold{\gamma}_{l}^{\alpha}=(\gamma_{l,1}^{\alpha},\gamma_{l,2}^{\alpha},\cdots,\gamma_{l,m}^{\alpha})^T$. Throughout, $T$ denotes transposition, and all components of the vector-valued function $\bold u_m=(u_1,u_2,\cdots,u_m)^T$ are assumed to lie in the Sobolev space $W_p^r(\R^d)$ with $r>d/p$. In the particular case $n=1$, $m=d$, and $P_1(D)=\nabla=(\partial_{x_1},\cdots,\partial_{x_d})^T$, constraint \eqref{physicalinvariants} reduces to $P_1(-D)\cdot \bold u_d=-\nabla\cdot \bold u_d=0$ (the divergence-free constraint \cite{Gaoetal4}). More generally, \eqref{physicalinvariants} covers the steady-state acoustics constraints \cite{Atteia1}, the divergence-curl constraints \cite{ChenandSuter}, the coupled divergence-free constraints of magnetohydrodynamics \cite{Li}, the curl-free constraint \cite{Fisher}, and many others.

To construct quasi-interpolation that \textbf{exactly} preserves the dot-product constraints \eqref{physicalinvariants}, we first build a matrix-valued kernel each of whose columns, viewed as a vector-valued function, \textbf{analytically} satisfies \eqref{physicalinvariants}. This is done in two steps. The first step specifies the kernel through its Fourier transform,
$$\widehat{\mathbb{K}}_{\ell,k,m,n,h}(\bm{\omega})=\Bigg[I_{m\times m }-P(-i\bm{\omega})G^\dagger(-i\bm{\omega})P^*(-i\bm{\omega})\Bigg]\widehat{\psi}_{\ell,k,h}(\bm{\omega}).$$  Here $I_{m\times m}$ is  an $m\times m$ identity matrix, $P(-i\bm{\omega})=[P_1(-i\bm{\omega}),P_2(-i\bm{\omega}),\cdots, P_n(-i\bm{\omega})]\in \mathbb{C}^{m\times n}$, $P^*(-i\bm{\omega})$  is the Hermitian conjugate of the matrix $P(-i\bm{\omega})$, namely, $P^*(-i\bm{\omega}):=P^T(i\bm{\omega})=[P_1^T(i\bm{\omega}),P_2^T(i\bm{\omega}),\cdots, P_n^T(i\bm{\omega})]^T\in \mathbb{C}^{n\times m}$, and $h$ is a positive tuning parameter. Furthermore, $G^\dagger(-i\bm{\omega})$ denotes the Moore--Penrose pseudo-inverse of the $n\times n$ matrix $G(-i\bm{\omega}):=P^*(-i\bm{\omega})P(-i\bm{\omega})$. We take the $h$-embedded  level-$k\ell$-type polyharmonic spline $\psi_{\ell,k,h}$ of Equation \eqref{polyharmonicspline} as a concrete choice, on account of its popularity and favourable properties, generalized polyharmonic splines \cite{Ortmann1} and other functions \cite{BuhmannBook} could be used equally well. In the second step, taking the inverse Fourier transform of $\widehat{\mathbb{K}}_{\ell,k,m,n,h}$ yields the kernel in the spatial domain in the convolution form
$\mathbb{K}_{\ell,k,m,n,h}=I_{m\times m}\psi_{\ell,k,h}-(P(D)\Phi_n)*(P^T(-D)\psi_{\ell,k,h})$, where $\Phi_n$ is defined in the distributional sense by the requirement $P(-i\bm\omega)\widehat{\Phi}_n(\bm\omega)P^*(-i\bm\omega)=P(-i\bm{\omega})G^\dagger(-i\bm{\omega})P^*(-i\bm{\omega}),\ \bm\omega\neq\bold 0$. In Particular, when $G^\dagger=G^{-1}$ is well-defined, then $\Phi_n$ is the generalized Green's matrix of the matrix differential operator $G(D)=P^T(-D)P(D)$.

Two features of our construction deserve emphasis. First, the scale parameter $h$ must be built into the polyharmonic spline $\psi_{\ell,k,h}$ itself, as this is what makes the operator transfer \eqref{interchange} legitimate, see Remark \ref{remarkonscale}. Second,  using the pseudo-inverse $G^\dagger$  rather than the inverse is essential: the columns of $P(-i\bm\omega)$ may   be linearly dependent at every frequency, as happens for the curl-type constraints   in Section $4$, and in such cases the ordinary inverse does not exist.

We next construct the convolution sequence $\mathbb{K}_{\ell,k,m,n,h}*\bold u_m$ and establish in Lemma \ref{matrixconvolution} the error estimates for the constrained vector field $\bold u_m$ satisfying dot-product constraints \eqref{physicalinvariants}, as well as for its derivatives. The proposed quasi-interpolant is derived by discretizing the aforementioned convolution sequence via a quadrature rule, which can be either deterministic or stochastic. We adopt the deterministic framework and employ the rectangular quadrature formula as a concrete realization, yielding a Schoenberg-type quasi-interpolant $Q_{\mathbb{K}_{\ell,k,m,n,h}}\bold u_m$, as formulated in Equation \eqref{quasiinterpolationofourpaper}. Corresponding $L^p$-error estimates for the quasi-interpolation approximation and its derivatives are further proven in Theorem \ref{errorestimateforderi}, Theorem \ref{errorestimateforderiLp}, Theorem \ref{errorestimateforderiL1}. Notably, the convolution sequence serves only as an auxiliary intermediate tool for bias-variance error analysis, the final quasi-interpolation scheme does not require any explicit convolution computation.

   As an application, we employ our structure-preserving quasi-interpolation approach to numerically decompose a general smooth vector-valued function $\bold u_3$ under the case  $m=3$, $n=1$.  Motivated from  Helmholtz--Hodge decomposition  \cite{Schwarz}, we first  propose a generalized Helmholtz--Hodge decomposition (see Lemma \ref{ghh}), which  decomposes a general smooth vector $\bold u_3$   into two parts  that respectively satisfies dot-product constraint \eqref{physicalinvariants} and cross-product constraint \begin{equation}\label{otherphysicalinvariants}
P_1(D)\times \bold u_3= \bold 0_3^T.
 \end{equation}
  We note that the cross-product constraint is a natural generalization of curl-free constraint and thus can be viewed as a duality of the dot-product constraint \eqref{physicalinvariants} with  $m=3$, $n=1$. Moreover, we can construct a matrix-valued kernel $$
 \mathbb{J}_{\ell,k,3,1,h}(\bold x):=(P_1(D)\Phi_1)* (P_1^T(-D)\psi_{\ell,k,h})(\bold x), \ \bold x\in \R^d,$$ such that it satisfies cross-product constraint \eqref{otherphysicalinvariants} column-wise.
 Then we construct two  convolution sequences  $\mathbb{J}_{\ell,k,3,1,h}*\bold u_3$ and  $\mathbb{K}_{\ell,k,3,1,h}*\bold u_3$ to approximate corresponding  parts of the generalized Helmholtz--Hodge decomposition and derive their  simultaneous error estimates in Lemma \ref{crosprojectionconvolutionerror} and Lemma \ref{dotprojectionconvolutionerror}. Finally, by discretizing these two convolution sequences, we get two  quasi-interpolation schemes $Q_{\mathbb{K}_{\ell,k,3,1,h}}\bold u_3$ and $Q_{\mathbb{J}_{\ell,k,3,1,h}}\bold u_3$ satisfying dot-product constraint \eqref{physicalinvariants} and cross-product constraint \eqref{otherphysicalinvariants}, respectively. Corresponding error estimates are derived in Theorem \ref{errorestimateforderi2}, which demonstrates that these two quasi-interpolation schemes together form a numerical decomposition of a general smooth vector field $\bold u_3$.

  The paper makes three main contributions. Firstly, leveraging orthogonal projection, we develop a general and  efficient framework to construct matrix-valued kernels   that \textbf{analytically} satisfy  multiple physical constraints. The resulting kernels contain the classical divergence-free and curl-free kernels as special cases, as well as those with tensor parameters corresponding to the pseudo-differential operators considered in \cite{Benbourhim1}, \cite{Bouhamidi}, \cite{Bouhamidi1}. Constructing kernels with prescribed properties has long been an active theme in approximation theory (see, e.g., \cite{Gaoetal1}, \cite{Schabackopti, Schabackernel, Schabackandwu}, \cite{GeneralizedWu}).  Different from existing studies, our work targets kernels subject to a system of differential constraints with arbitrary, possibly unequal orders and a genuinely singular Gram symbol.  Secondly, we establish a unified theory of vectorial quasi-interpolation that \textbf{exactly} preserves intrinsic physical structures of the approximand. It yields a structure-preserving approximation directly without solving any constrained optimization problem. Thirdly, we propose a generalized Helmholtz--Hodge decomposition paired with structure-preserving quasi-interpolation schemes, which enables a  quasi-interpolation-based  approach for numerical decomposition of general smooth vector fields.

 The remainder of the paper is organized as follows. Section $2$ has two parts: the first one reviews notations and definitions, while the second one develops a construction of $h$-embedded level-$k\ell$-type polyharmonic splines and studies their properties in the spatial and frequency domains. Section $3$ details the construction of structure-preserving vectorial quasi-interpolation schemes and their application to the numerical decomposition of a general smooth vector field. Section $4$ presents concrete examples of matrix-valued kernels for various physical constraints; with these kernels one obtains quasi-interpolation schemes that \textbf{exactly} preserve the corresponding physical structures of the target function. Section $5$ comprises two numerical validations: Subsection $5.1$ presents numerical simulations that verify the theoretical convergence rates and the exact structure-preserving property of the proposed schemes, while Subsection $5.2$ illustrates the capability of the developed quasi-interpolation method for numerical vector field decomposition. Conclusions and  discussions are given in Section $6$, and a technical proof is collected in the Appendix.

\section{\textit{Preliminaries and notations}}
This section collects the basic material used for constructing our quasi-interpolation and deriving its error estimates. We begin with notation and definitions.
\subsection{\textit{Notation and definitions}}
Let  $1 \leq p \leq \infty$ and $0\leq r\leq \infty$ be two given constants.  Denote by  $W_p^r(\R^d)$ the Sobolev space, which, for $r=0$, is just the  Lebesgue space  consisting of functions $f$  satisfying $\|f\|_{p} < \infty$. Throughout we assume $r>d/p$, so that $W_p^r(\R^d)$ is embedded into a space of continuous functions. The convolution of a function $f$ with a function $g$ (or of the distributions associated to the functions $f$ and $g$) is defined via
$f*g(\bold x)=\int_{\R^d}f(\bold y)g(\bold x-\bold y)d\bold y$.
   Moreover, by letting $\mathfrak{F}(f)$  be the (generalized) Fourier transform of $f$ defined (in the distributional sense \cite{Stein}) as
$$\mathfrak{F}(f):=\hat{f}(\bm{\omega})=(2\pi)^{-d/2}\int_{\mathbb{R}^d}f(\bold x)e^{-i\bm{\omega}\cdot \bold x} d\bold x,\ \text{where}\ \bm{\omega}\cdot \bold x=\sum_{j=1}^d \omega_j x_j,$$ then we have the celebrated identity $\mathfrak{F}(f*g)=\hat{f}\hat{g}$.

Denote by $(W_p^r(\R^d))^{m}$ and $(W_p^r(\R^d))^{n\times m}$ the vector-valued and matrix-valued Sobolev space consisting of elements whose components are all in $W_p^r(\R^d)$, respectively. For $\bold u_m=(u_1,u_2,\cdots, u_m)^T\in (W_p^r(\R^d))^{m}$, we define its norm via $\|\bold u_m\|_p=\sum_{j=1}^m\|u_j\|_p$, while  for an ${n\times m}$ matrix-valued  function $\bold F=(f_{jk})$, we  define its norm as $\|\bold F\|_p=\sum_{j,k}\|f_{jk}\|_p$. In addition, we define the convolution of a matrix-valued function with a vector-valued function as $$(\bold F*\bold u_m)(\bold x)=\int_{\R^d}\bold F(\bold y)\bold u_m(\bold x-\bold y)d\bold y=\bold h(\bold x)=(h_1(\bold x), h_2(\bold x),\cdots, h_n(\bold x))^T$$ with $h_j(\bold x)=\sum_{k=1}^m f_{jk}*u_k(\bold x)$ for $1\leq j\leq n$. Likewise, the convolution of two matrix-valued functions is defined by $$(\bold F*\bold G)(\bold x)=\int_{\R^d}\bold F(\bold y)\bold G(\bold x-\bold y)d\bold y=\bold H(\bold x)=(h_{ij}(\bold x))$$ with $h_{ij}(\bold x)=\sum_{k=1}^m f_{ik}*g_{kj}(\bold x)$.
 As usual, the Fourier transform operator $\mathfrak{F}$ is applied componentwise to matrix- and vector-valued functions. Then we have $\mathfrak{F}(\bold F*\bold u_m)=\mathfrak{F}(\bold F)\cdot\mathfrak{F}(\bold u_m)$ and $\mathfrak{F}(\bold F*\bold G)=\mathfrak{F}(\bold F)\cdot\mathfrak{F}(\bold G)$.
\subsection{\textit{(Generalized) thin plate splines}}
Thin plate splines were introduced by Duchon in $1977$ \cite{Duchon} as a multivariate analogue of univariate splines. Denote by $\nabla=(\partial_{x_1},\cdots,\partial_{x_d})^T$ the gradient operator and $\Delta=\nabla^T\nabla$ the Laplace operator. Let $\ell$ be a positive integer and $\Delta^{\ell}$ be the $\ell$th iterate of $\Delta$. The thin plate spline $\phi_{\ell}$ is defined as   a fundamental solution of $\Delta^{\ell}$  in a distributional sense, that is, $\Delta^{\ell}\phi_{\ell}=\delta$.   In particular, we have
\begin{equation}\label{fundamentalsolution}
 \phi_{\ell}(\bold x)=
\begin{cases}
E_{\ell,d}\|\bold x\|^{2\ell-d}, \text{for an odd}\ d,\\
E_{\ell,d}\|\bold x\|^{2\ell-d}\ln\|\bold x\|, \ \text{for an even}\ d,
\end{cases}
\end{equation}
where $$E_{\ell,d}=\frac{\Gamma(d/2)}{2^{\ell}\pi^{d/2}(\ell-1)!\prod_{j=0,j\neq \ell-d/2}^{\ell-1}(2\ell-2j-d)}.$$
Moreover, taking finite linear combinations of integer translates of $\phi_{\ell}$, Rabut \cite{Rabut0} first constructed the level-$k\ell$ polyharmonic splines $\psi_{\ell,k}(\bold x)$ for $k\leq \ell$ and then used their scaled counterparts $h^{-d}\psi_{\ell,k}(\bold x/h)$ to build quasi-interpolants \cite{Rabut0, Rabut, Rabut1}. We extend this idea to construct $h$-embedded level-$k\ell$-type polyharmonic splines $\psi_{\ell,k,h}(\bold x)$ in which the scale parameter is built into the spline itself, so that they can be used directly for quasi-interpolation. As will become clear in Remark \ref{remarkonscale}, this reformulation is in fact indispensable for our construction.

Define the $j^{\text{th}}$-direction component $(\widetilde{\Delta_h})_j$ of the central difference operator $\widetilde{\Delta_h}$ of step-size $h$ by \begin{equation*}
(\widetilde{\Delta_h})_jf=f(\cdot -\bold e_jh)-2f(\cdot)+f(\cdot+\bold e_jh), \quad j=1,2,\cdots,d.\end{equation*} One verifies at once that $h^{-2}\widetilde{\Delta_h}f$ converges to $\Delta f$ as $h$ tends to zero. Following Rabut \cite{Rabut0}, we then define the $h$-embedded level-$k\ell$-type polyharmonic spline as
\begin{equation}\label{polyharmonicspline}
\psi_{\ell,k,h}(\bold x)=(-1)^{\ell}h^{-2\ell}q_{d,\ell,k}(\widetilde{\Delta_h})\phi_{\ell}(\bold x), \quad  \bold x\in \R^d.
\end{equation}
Here $q_{d,\ell,k}$ is an $(\ell+k)$th-order truncation of  a polynomial $p_{d,\ell,k}$ (of degree $k\ell$ with respect to $\widetilde{\Delta_h}$)  defined as
$$p_{d,\ell,k}(\widetilde{\Delta_h})=\Bigg(\sum_{i=0}^{k-1}a_i\Bigg(\sum_{j=1}^d(\widetilde{\Delta_h})_j^{i+1}\Bigg)\Bigg)^{\ell},\ \text{where}\  a_i=(-1)^i2(i!)^2/(2i+2)!.$$ Taking $h=1$ recovers the level-$k\ell$ polyharmonic spline $\psi_{\ell,k}$. More importantly, for general $h$ the following two identities hold (see  \ref{app:scaling} for a detailed proof): \begin{equation}\label{keyrelation}
\psi_{\ell,k,h}(\bold x)=h^{-d}\psi_{\ell,k}(\bold x/h), \ \widehat{\psi}_{\ell,k,h}(\bm\omega)=\widehat{\psi_{\ell,k}}(h\bm\omega).
\end{equation}

Consequently, $\psi_{\ell,k,h}$ can be characterized in both the spatial and the frequency domain.
\begin{lemma}\label{polyharmonicsplinepropert}
Let $\psi_{\ell,k,h}$ be defined in Equation \eqref{polyharmonicspline}.  Then we have $\psi_{\ell,k,h}\in C^{2\ell-d-1}(\R^d)$  and
\begin{equation*}D^{\alpha}\psi_{\ell,k,h}(\bold x)=\mathcal{O}(h^{2k-2|\alpha|}\|\bold x\|^{-d-2k+|\alpha|})
\end{equation*}
holds true  for any $0\leq |\alpha|< 2k$  as  $\|\bold x\|$ tends to infinity. In addition, the Fourier transform of $\psi_{\ell,k,h}$ reads \begin{equation*}
      \widehat{\psi}_{\ell,k,h}(\bm{\omega})=\frac{q_{d,\ell,k}(-4\sin^2(h\bm{\omega}/2))}{\|h\bm{\omega}\|^{2\ell}}
      \end{equation*}
and  thus satisfies the conditions   \begin{equation}\label{Strangfix}
\begin{cases}
\widehat{\psi}_{\ell,k,h}\in C^{2k}(\R^d),\\
D^{\alpha}\widehat{\psi}_{\ell,k,h}(\bold 0)=\delta_{0,\alpha},\ 0\leq |\alpha|< 2k,\\
D^{\alpha}\widehat{\psi}_{\ell,k,h}(2\pi \bold j/h)=0, \  0\leq |\alpha|< 2k, \bold j\in \mathbb{Z}^d/\{\bold 0\}.
\end{cases}
\end{equation}
\end{lemma}
\begin{proof}
This is a direct consequence of  Proposition $3$ in Rabut \cite{Rabut0}  together with Identity \eqref{keyrelation}.
\end{proof}
In addition,   we have   the following simultaneous error estimates based on Theorem $2.1$ in reference \cite{Leiandjia}.
 \begin{lemma}\label{convolution}
   Let $\psi_{\ell,k,h}*f$ be a convolution sequence with $\psi_{\ell, k,h}$ being the level-$k\ell$-type polyharmonic spline defined in Equation \eqref{polyharmonicspline}. Then, for any $ f\in W^r_p (\R^d)$, we have the simultaneous error estimates
$$\|\psi_{\ell,k,h}*(D^{\alpha}f)-D^{\alpha}f\|_p = \mathcal{O}(h^{2k})$$
  for  $0\leq |\alpha|< \min\{2k, r-2k\}$ with $r-2k>d/2$.
 \end{lemma}

Furthermore, the corresponding  Schoenberg's model \cite{Schoenberg} $Q_{\ell,h}f(\bold x)=h^d\sum_{\bold j\in \mathbb{Z}^d}f(\bold j h)\psi_{\ell,k,h}(\bold x-\bold j h)$ provides simultaneous approximations  to both $f$ and its derivatives. Corresponding approximation errors can be derived as  \cite{Leiandjia} \begin{equation}\label{simultaneousofpol}
\|D^{\alpha}Q_{\ell,h}f-D^{\alpha}f\|_{p}=\mathcal{O}(h^{2k-|\alpha|})
\end{equation} for any $0\leq |\alpha|<\min\{2k,r-2k,2\ell -d-1\}$ with $r-2k>d/2$.

Estimate \eqref{simultaneousofpol} shows that approximating high-order derivatives of $f$ by those of $Q_{\ell,h}f$ forces one to use higher-order polyharmonic splines $\psi_{\ell,k,h}$, which increases the complexity of the scheme and may impair its stability. Moreover, since $\phi_{\ell}$ contains a logarithmic term for even $d$, a continuous extension around the origin is required. To circumvent these difficulties, Ortmann and Buhmann \cite{Ortmann1} recently proposed a new class of infinitely smooth generalized thin plate splines of the form
 \begin{equation*}
 \phi_{\ell, c}(\bold x)=
\begin{cases}
E_{\ell,d}\sqrt{c^{4\ell-2d}+\|\bold x\|^{4\ell-2d}}, \text{for an odd}\ d,\\
E_{\ell,d}(c^{2\ell-d}+\|\bold x\|^{2\ell-d})\ln(c^{2\ell-d}+\|\bold x\|^{2\ell-d}), \ \text{for an even}\ d,
\end{cases}
\end{equation*}
  where $c$ is a small positive shape parameter. We can replace $\phi_{\ell}$ in Equation \eqref{polyharmonicspline}  with $\phi_{\ell, c}$ to construct a level-$k\ell$ generalized polyharmonic spline
\begin{equation}\label{generalizedpolyharmonicspline}
\psi_{\ell,k,c,h}(\bold x)=(-1)^{\ell}h^{-2\ell}q_{d,\ell,k}(\widetilde{\Delta_h})\phi_{\ell,c,h}(\bold x), \quad  \bold x\in \R^d.
\end{equation} One may expect the resulting Schoenberg model $Q_{\ell,c,h}f(\bold x)=h^d\sum_{\bold j\in \mathbb{Z}^d}f(\bold j h)\psi_{\ell,k,c,h}(\bold x-\bold jh)$ to provide simultaneous approximations to $f$ and its derivatives as well, in analogy with classical multiquadric (MQ) quasi-interpolation \cite{MaandWu}. More importantly, since $\psi_{\ell,k,c,h}$ is infinitely smooth, $D^{\alpha}Q_{1,c,h}f$ may already approximate $D^{\alpha} f$ for a suitable choice of $c$ and $h$. We leave the corresponding error estimates, and the stability of numerical differentiation based on such a scheme, as an interesting direction for future work.
\section{\textit{ The main results}}
This section consists of two parts. The first part concerns the reconstruction, from discrete function values, of a vector-valued function whose physical structure is described by the dot-product constraints \eqref{physicalinvariants}. We develop a general theory of quasi-interpolation that \textbf{exactly} preserves this structure, and derive simultaneous approximation orders for the target function and its derivatives. The second part proposes a quasi-interpolation based method for numerically decomposing a general smooth vector field: we first extend the classical Helmholtz--Hodge decomposition to a general setting, splitting a smooth vector field into two components that satisfy the dot-product constraint \eqref{physicalinvariants} and the cross-product constraint \eqref{otherphysicalinvariants} respectively, and then construct a structure-preserving quasi-interpolation scheme for each component.
\subsection{\textit{Quasi-interpolation for vector-valued functions with prescribed physical structures}}
   The literature offers three main perspectives on quasi-interpolation: the quasi-Lagrange perspective \cite{BuhmannBook}, the moving least squares perspective \cite{Fasshauer1}, and the semi-discrete convolution perspective \cite{GeneralizedWu}. The last of these regards a quasi-interpolant as the discretization of a convolution sequence obtained by convolving the target function with a delta-like sequence. Its construction thus splits into two ingredients: the delta-like sequence and the quadrature rule. This viewpoint also decomposes the approximation error naturally into an approximation term and a discretization term, mirroring the bias--variance decomposition of statistical learning theory, and it makes the optimality and regularization properties of quasi-interpolation with a tunable scale parameter accessible \cite{Gaoetal2}. We adopt this perspective here, and begin by constructing matrix-valued kernels whose columns \textbf{analytically} satisfy the physical constraints of the target function.
\subsubsection{\textit{ Matrix-valued kernels with physical constraints}}
We first introduce an operator defined through the generalized Fourier transform: \begin{equation}\label{proop}
\mathbb{P}(\bm \omega)=I_{m\times m}-P(-i\bm{\omega})G^\dagger(-i\bm{\omega})P^*(-i\bm{\omega}).
 \end{equation}
Then we have the following proposition characterizing properties of $\mathbb{P}$.
\begin{proposition}\label{propofp}
Let $G^\dagger(-i\bm\omega)$ be the Moore--Penrose pseudo-inverse of $G(-i\bm\omega):=P^*(-i\bm\omega)P(-i\bm\omega)$. Denote by $V=\text{Ran} P$  the range of the operator $P$. Then $\mathbb{P}(\bm\omega)$ defined in Equation \eqref{proop} is an orthogonal projector onto  the orthogonal complement $V^{\perp}$ of $V$. In particular, $\mathbb{P}(\bm\omega)$ is a contraction,
\begin{equation}\label{contraction}
\|\mathbb{P}(\bm\omega)\bold v\|\leq \|\bold v\|\quad \text{for all}\ \bold v\in \mathbb{C}^m\ \text{and all}\ \bm\omega\neq \bold 0.
\end{equation}
Moreover, the annihilation $P_l(i\bm\omega)\cdot\mathbb{P}(\bm\omega) = \bold 0_m^T$  holds for all $1\leq l\leq n$.
\end{proposition}

\begin{proof}
 We first show that the matrix $PG^\dagger P^*$ is Hermitian and idempotent. Using the Penrose identity $G^\dagger GG^\dagger= G^\dagger$, we have
\[
(PG^\dagger P^*)^2 = P\,\underbrace{G^\dagger P^*P\,G^\dagger}_{=\,G^\dagger GG^\dagger \,=\, G^\dagger}\,P^* = PG^\dagger P^*.
\]
 Thus $PG^\dagger P^*$ is idempotent.  Hermiticity follows directly from construction, as $G=P^*P$ is Hermitian and the Moore--Penrose pseudo-inverse of a Hermitian matrix remains Hermitian. Moreover, we have  $\text{Ran}(PG^\dagger P^*) = \text{Ran}P$ since $G^\dagger$ maps $\text{Ran}P^* = \text{Ran}G$ onto itself bijectively. Hence $PG^\dagger P^*$ is the orthogonal projector onto $V$ and accordingly $\mathbb{P} = I- PG^\dagger P^*$ is the one  onto the orthogonal complement $V^\perp$. We now verify the annihilation property. Recall $P_l^T(i\bm\omega) = \bold e_l^TP^*(-i\bm\omega)$. Substitute it into $\bold e_l^TP^*\mathbb{P}$:
\[
\bold e_l^TP^*\mathbb{P} = \bold e_l^T\big(P^* - P^*PG^\dagger P^*\big) = \bold e_l^T\big(P^* - GG^\dagger P^*\big) = \bold e_l^T\big(P^* - P^*\big) = \bold 0_m^T,
\]
where $GG^\dagger P^* = P^*$ holds because $GG^\dagger$ is the orthogonal projector onto $\text{Ran}G = \text{Ran}P^*$. Finally, \eqref{contraction} is immediate: being Hermitian and idempotent, $\mathbb{P}$ satisfies $\|\mathbb{P}\bold v\|^2=\bold v^*\mathbb{P}^*\mathbb{P}\bold v=\bold v^*\mathbb{P}\bold v\leq \|\bold v\|\,\|\mathbb{P}\bold v\|$ by the Cauchy--Schwarz inequality, whence $\|\mathbb{P}\bold v\|\leq \|\bold v\|$.
\end{proof}
 With the projector $\mathbb{P}$ and the spline $\psi_{\ell,k,h}$ at hand, we characterize our matrix-valued kernel in the frequency domain by $\widehat{\mathbb{K}}_{\ell,k,m,n,h}(\bm{\omega})=\mathbb{P}(\bm{\omega})\widehat{\psi}_{\ell,k,h}(\bm{\omega})$. Taking the inverse Fourier transform of both sides gives \begin{equation}\label{matrixkernel}
 \mathbb{K}_{\ell,k,m,n,h}(\bold x):=I_{m\times m}\psi_{\ell,k,h}(\bold x) -(P(D)\Phi_n)* (P^T(-D)\psi_{\ell,k,h})(\bold x), \ \bold x\in \R^d,
  \end{equation}
 where the $n\times n$ matrix-valued function $\Phi_n$ is defined in the distributional sense by the requirement \begin{equation}\label{defofphin}
 P(-i\bm\omega)\widehat{\Phi}_n(\bm\omega)P^*(-i\bm\omega)=P(-i\bm{\omega})G^\dagger(-i\bm{\omega})P^*(-i\bm{\omega}), \quad \bm\omega\neq \bold 0.
 \end{equation}
 We stress that \eqref{defofphin} constrains $\widehat{\Phi}_n$ only through the above sandwich product, and hence leaves $\widehat{\Phi}_n$ free on the kernel of $P(-i\bm\omega)$. The obvious choice $\widehat{\Phi}_n=G^\dagger$ is admissible, but simpler ones are often available, as the examples of Section $4$ illustrate. We next verify that $\mathbb{K}_{\ell,k,m,n,h}$ \textbf{analytically} satisfies the dot-product constraints \eqref{physicalinvariants} column-wise.
  \begin{lemma}\label{physicalconstraintsofkernels}
Let $\Phi_n$ satisfy \eqref{defofphin} and let $\mathbb{K}_{\ell,k,m,n,h}$ be defined in Equation \eqref{matrixkernel}. Then the column-wise dot product satisfies
$P_l(-D)\cdot\mathbb{K}_{\ell,k,m,n,h}=P_l^T(-D)\mathbb{K}_{\ell,k,m,n,h}=\bold 0_m$ for all $1\leq l\leq n$.
\end{lemma}
\begin{proof}
By \eqref{defofphin}, the Fourier transform of \eqref{matrixkernel} is $\widehat{\mathbb{K}}_{\ell,k,m,n,h}(\bm{\omega})=\mathbb{P}(\bm \omega)\widehat{\psi}_{\ell,k,h}(\bm{\omega})$. Proposition \ref{propofp} then gives $P_l(i\bm{\omega})\cdot \widehat{\mathbb{K}}_{\ell,k,m,n,h}(\bm{\omega})=\big(P_l(i\bm{\omega})\cdot\mathbb{P}(\bm{\omega})\big)\widehat{\psi}_{\ell,k,h}(\bm{\omega})=\bold 0_m$ for all $1\leq l\leq n$. Taking inverse Fourier transforms yields $P_l(-D)\cdot\mathbb{K}_{\ell,k,m,n,h}=P_l^T(-D)\mathbb{K}_{\ell,k,m,n,h}=\bold 0_m$ for all $1\leq l\leq n$.
\end{proof}
\begin{remark}
The projector $\mathbb{P}$ is the core of our construction. In principle, any function of the form $\mathbb{P}(\bm{\omega})\Psi(\bm\omega)$ in the frequency domain yields a matrix-valued kernel satisfying the dot-product constraints \eqref{physicalinvariants}, for a suitably chosen $\Psi$. We take $\Psi=\widehat{\psi}_{\ell,k,h}$ as a concrete choice because of its favourable properties and popularity.
\end{remark}
We next form the convolution sequence $$\mathbb{K}_{\ell,k,m,n,h}*(D^{\alpha}\bold u_m)(\bold x):=\int_{\R^d}\mathbb{K}_{\ell,k,m,n,h}(\bold x-\bold t)D^{\alpha}\bold u_m(\bold t)d\bold t$$ and derive its  error estimates.

 \begin{lemma}\label{matrixconvolution}
Let $\mathbb{K}_{\ell,k,m,n,h}*(D^{\alpha}\bold u_m)$ be defined as above.  Then, for any $ \bold u_m \in (W^r_p (\R^d))^m$ satisfying physical constraints  \eqref{physicalinvariants},  we have    error estimates
\begin{equation}\label{matxconvoerr}
\|\mathbb{K}_{\ell,k,m,n,h}*(D^{\alpha}\bold u_m)-D^{\alpha}\bold u_m\|_p = \mathcal{O}(h^{2k})
\end{equation}   for any $0\leq |\alpha|<\min\{2k, r-2k\}$.
 \end{lemma}
\begin{proof}
 Setting $\bold v_m^{\alpha}=D^{\alpha}\bold u_m$,  we have $P_l(-D)\cdot\bold v_m^{\alpha}=0$ for any $1\leq l\leq n$ and any $0\leq |\alpha|\leq \min\{2k, r-2k\}$. Moreover, we can rewrite the convolution sequence as
$$\mathbb{K}_{\ell,k,m,n,h}*(D^{\alpha}\bold u_m)(\bold x)= (I_{m\times m}\psi_{\ell,k,h})*\bold v_m^{\alpha}(\bold x)-(P(D)\Phi_n)*(P^T(-D)\psi_{\ell,k,h})*\bold v_m^{\alpha}(\bold x).$$
This together with the observation that
\begin{equation}\label{interchange}(P(D)\Phi_n)*(P^T(-D)\psi_{\ell,k,h})*\bold v_m^{\alpha}(\bold x)=(P(D)\Phi_n)*\psi_{\ell,k,h}*(P^T(-D)\bold v_m^{\alpha})(\bold x)=(P(D)\Phi_n)*\psi_{\ell,k,h}*\bold 0_n=\bold 0_m
\end{equation}
leads to
\begin{equation*}
\begin{split}\mathbb{K}_{\ell,k,m,n,h}* \bold v_m^{\alpha}(\bold x)
&=I_{m\times m}\psi_{\ell,k,h}*\bold v_m^{\alpha}(\bold x)\\
&=(\psi_{\ell,k,h}*(D^{\alpha}u_1)(\bold x),\cdots,\psi_{\ell,k,h}*(D^{\alpha}u_m)(\bold x))^T\\
&=\psi_{\ell,k,h}*\bold v_m^{\alpha}(\bold x).
\end{split}
\end{equation*}
Therefore, we have $$\|\mathbb{K}_{\ell,k,m,n,h}*(D^{\alpha}\bold u_m)-D^{\alpha}\bold u_m\|_p =\|\psi_{\ell,k,h}*\bold v_m^{\alpha}-\bold v_m^{\alpha}\|_p:=\sum_{j=1}^m\|\psi_{\ell,k,h}*(D^{\alpha}u_j)-(D^{\alpha}u_j)\|_p,$$
which together with Lemma \ref{convolution} yields Equation \eqref{matxconvoerr}.
\end{proof}
\begin{remark}\label{remarkonscale}
Our construction of $h$-embedded level-$k\ell$-type polyharmonic spline $\psi_{\ell,k,h}$ is what makes Equation \eqref{interchange} legitimate. Because the scale parameter $h$ is built into $\psi_{\ell,k,h}$ itself, the expression $\mathbb{K}_{\ell,k,m,n,h}*\bold v_m^{\alpha}$ is a genuine convolution of undilated factors, so that the constant-coefficient operator $P^T(-D)$ may be transferred from one factor to the other according to the rule $(L g)*w=g*(Lw)$ for a linear operator $L$. If instead $\psi_{\ell,k,h}(\bold x)$ were replaced by $h^{-d}\psi_{\ell,k}(\bold x/h)$ in the convolution sequence, then \eqref{interchange} would fail: the operator $P^T(-D)$ cannot be moved from the kernel $h^{-d}\psi_{\ell,k}((\bold x-\bold t)/h)$ onto $\bold v_m^{\alpha}$, because the dilation $(\bold x-\bold t)/h$ couples the differentiation variable to $\bold t$ and produces uncontrolled factors $h^{-1}$ under the chain rule.
\end{remark}
 In practice only discrete function values at sampling centers are available, so the convolution sequence must be discretized to obtain a semi-discrete counterpart, that is, a quasi-interpolant.
\subsubsection{\textit{Structure-preserving   quasi-interpolation}}
For the purposes of the theoretical analysis we focus on quasi-interpolation in Schoenberg's form, based on discrete function values at equidistant sampling centers throughout $\R^d$. Low-discrepancy centers \cite{Gaoetal2}, random centers \cite{Gaoetal3}, and other choices \cite{Krieg} may be used instead. Over a bounded domain $\Omega\subset \R^d$ one must address the well-known boundary effect: applying a quasi-interpolant designed for the whole space directly to data sampled from a bounded domain gives poor approximation near the boundary. A common remedy is boundary extension, see for instance \cite{EPSTEIN}, \cite{ZMRS}; more recently, the authors of \cite{Gaoetal4} proposed a Boolean sum of interpolation near the boundary and quasi-interpolation in the interior. We leave such a topic for future work.

 Given sampling data $\{(\bold jh,\bold u_m(\bold jh))\}_{\bold j\in \mathbb{Z}^d}$  with equidistant sampling centers $\{\bold jh\}_{\bold j\in \mathbb{Z}^d}$ over the whole space $\R^d$, we  first construct an ansatz
\begin{equation}\label{quasiinterpolationofourpaper}
Q_{\mathbb{K}_{\ell,k,m,n,h}}\bold u_m(\bold x):=h^d\sum_{\bold j\in \mathbb{Z}^d}\mathbb{K}_{\ell,k,m,n,h}(\bold x-\bold jh)\bold u_m(\bold jh),\ \bold x\in \R^d,
\end{equation}
following Schoenberg's model \cite{Schoenberg}. Then show that $Q_{\mathbb{K}_{\ell,k,m,n,h}}\bold u_m$ \textbf{exactly} preserves the physical constraints \eqref{physicalinvariants}. Applying the vector differential operator $P_l(-D)$ column-wise to $Q_{\mathbb{K}_{\ell,k,m,n,h}}\bold u_m$ gives
$$P_l(-D) \cdot\big(Q_{\mathbb{K}_{\ell,k,m,n,h}}\bold u_m(\bold x)\big)=h^d\sum_{\bold j\in \mathbb{Z}^d}\big(P_l(-D) \cdot \mathbb{K}_{\ell,k,m,n,h}(\bold x-\bold jh)\big) \bold u_m(\bold jh),$$ which together with Lemma \ref{physicalconstraintsofkernels} yields
$$P_l(-D) \cdot\big(Q_{\mathbb{K}_{\ell,k,m,n,h}}\bold u_m(\bold x)\big)=h^d\sum_{\bold j\in \mathbb{Z}^d}\bold 0_{m}^T \bold u_m(\bold jh)=0, \qquad 1\leq l\leq n.$$
It remains to show that $Q_{\mathbb{K}_{\ell,k,m,n,h}}\bold u_m$ is indeed a quasi-interpolant  by deriving its simultaneous error estimates. Throughout, we set $\bold v_m^{\alpha}:=D^{\alpha}\bold u_m$ and write $\mathbb{K}:=\mathbb{K}_{\ell,k,m,n,h}$, $Q_{\mathbb{K}}:=Q_{\mathbb{K}_{\ell,k,m,n,h}}$ for brevity.
\begin{theorem}\label{errorestimateforderi}
Let  $\bold u_m\in (W_p^r(\R^d))^m$ be a vector-valued function having physical structures characterized via dot-product constraints \eqref{physicalinvariants}. Let $Q_{\mathbb{K}}\bold u_m$ be defined  in Equation \eqref{quasiinterpolationofourpaper} with $\mathbb{K}$ being constructed in Equation \eqref{matrixkernel}.   Then, for $p=2$ and $p=\infty$, we have simultaneous error estimates
$$\|D^{\alpha}(Q_{\mathbb{K}}\bold u_m)-D^{\alpha}\bold u_m\|_{p}=\mathcal{O}(h^{2k-|\alpha|})$$  for  any $0\leq |\alpha|< \min\{2k, r-2k,2\ell-d-1\}$ with $r-2k>d/2$.
\end{theorem}
\begin{proof}
 We decompose the approximation error into a bias and a variance part,
\begin{equation}\label{biasvariance}
D^{\alpha}(Q_{\mathbb{K}}\bold u_m)-D^{\alpha}\bold u_m
=\underbrace{\Big[D^{\alpha}(Q_{\mathbb{K}}\bold u_m)-\mathbb{K}*\bold v_m^{\alpha}\Big]}_{=:E_1\ \text{(variance)}}
+\underbrace{\Big[\mathbb{K}*\bold v_m^{\alpha}-\bold v_m^{\alpha}\Big]}_{=:E_2\ \text{(bias)}} .
\end{equation}
By Lemma \ref{matrixconvolution}, $\|E_2\|_{p}=\mathcal{O}(h^{2k})$ for every $1\leq p\leq \infty$, so it remains to bound $\|E_1\|_{p}$ for $p=2$ and $p=\infty$.

Since $|\alpha|<2k$, Lemma \ref{polyharmonicsplinepropert} gives $D^{\alpha}\mathbb{K}(\bold x)=\mathcal{O}(\|\bold x\|^{-d-2k+|\alpha|})$ with $-d-2k+|\alpha|<-d$, so the lattice sum $\sum_{\bold j}|D^{\alpha}\mathbb{K}(\bold x-\bold jh)|$ converges uniformly in $\bold x$, and $r>d/p$ embeds $\bold u_m$ into a bounded continuous function. Hence the series defining $Q_{\mathbb{K}}\bold u_m$ may be differentiated term by term, and the Poisson summation formula applies to $\bold t\mapsto (D^{\alpha}\mathbb{K})(\bold x-\bold t)\bold u_m(\bold t)$. Its $\bold j=\bold 0$ frequency contribution is exactly $\mathbb{K}*\bold v_m^{\alpha}(\bold x)$, so only the nonzero aliases survive in $E_1$:
\begin{equation}\label{aliasrep}
E_1(\bold x)=(2\pi)^{-d/2}\int_{\R^d}\sum_{\bold j\in \mathbb{Z}^d/\{\bold 0\}}\big(i(\bm{\omega}+2\bold j\pi/h)\big)^{\alpha}\widehat{\mathbb{K}}(\bm{\omega}+2\bold j\pi/h)\,\widehat{\bold u_m}(\bm{\omega})\,e^{i\bold x\cdot\bm{\omega}}e^{2i\bold j\pi\cdot \bold x/h}\,d\bm{\omega}.
\end{equation}

Substituting $\bm\zeta=\bm\omega+2\bold j\pi/h$ shows that the $\bold j$th summand of \eqref{aliasrep} is the inverse Fourier transform of $\bm\zeta\mapsto (i\bm\zeta)^{\alpha}\widehat{\mathbb{K}}(\bm\zeta)\widehat{\bold u_m}(\bm\zeta-2\bold j\pi/h)$. Hence, by the triangle inequality in $L^2$ followed by the Parseval--Plancherel identity and the substitution back to $\bm\omega$, we have
\begin{equation}\label{L2split}
\|E_1\|_{2}\ \leq\ \sum_{\bold j\in \mathbb{Z}^d/\{\bold 0\}}\Big\|(\bm{\omega}+2\bold j\pi/h)^{\alpha}\,\widehat{\mathbb{K}}(\bm{\omega}+2\bold j\pi/h)\,\widehat{\bold u_m}(\bm{\omega})\Big\|_{L^2(d\bm\omega)} .
\end{equation}

Moreover, since $\widehat{\mathbb{K}}=\mathbb{P}\widehat{\psi}_{\ell,k,h}$ and $\mathbb{P}$ is a contraction by \eqref{contraction}, the matrix factor may be discarded pointwise in $\bm\omega$:
$$\big|\widehat{\mathbb{K}}(\bm\zeta)\widehat{\bold u_m}(\bm\omega)\big|\leq \big|\widehat{\psi}_{\ell,k,h}(\bm\zeta)\big|\,\big|\widehat{\bold u_m}(\bm\omega)\big| .$$   Combining this with $\widehat{\psi}_{\ell,k,h}(\bm\omega)=\widehat{\psi_{\ell,k}}(h\bm\omega)$ from \eqref{keyrelation} and with $(\bm{\omega}+2\bold j\pi/h)^{\alpha}=h^{-|\alpha|}(h\bm{\omega}+2\bold j\pi)^{\alpha}$, the right-hand side of \eqref{L2split} is at most
\begin{equation}\label{afterrescale}
h^{-|\alpha|}\sum_{\bold j\in \mathbb{Z}^d/\{\bold 0\}}\Big\|(h\bm{\omega}+2\bold j\pi)^{\alpha}\,\widehat{\psi_{\ell,k}}(h\bm{\omega}+2\bold j\pi)\,\widehat{\bold u_m}(\bm{\omega})\Big\|_{L^2(d\bm\omega)} .
\end{equation}
Furthermore, the Taylor expansion of $\widehat{\psi_{\ell,k}}$ about $2\bold j\pi$ in the increment $h\bm\omega$ produces an intermediate point on the segment joining $\bold 0$ to $h\bm\omega$, and the uniform bound of \cite{Rabut0} recalled in \eqref{rabutsup} below is asserted only for increments of length at most one. We therefore split $\R^d=\Omega_{\text{in}}\cup\Omega_{\text{out}}$ with $\Omega_{\text{in}}:=\{\bm\omega: \|h\bm\omega\|\leq 1\}$ and $\Omega_{\text{out}}:=\{\bm\omega:\|h\bm\omega\|>1\}$, and treat the two regions differently.

\emph{On $\Omega_{\text{in}}$.} Here the increment does lie in the admissible range. Since $D^{\beta}\widehat{\psi_{\ell,k}}(2\bold j\pi)=0$ for $0\leq |\beta|<2k$ and $\bold j\neq \bold 0$ by \eqref{Strangfix}, Taylor's theorem with Lagrange remainder gives
\begin{equation}\label{taylorstep}
\widehat{\psi_{\ell,k}}(h\bm{\omega}+2\bold j\pi)=\frac{h^{2k}\bm{\omega}^{2k}}{(2k)!}\widehat{\psi_{\ell,k}}^{(2k)}(\bm{\xi}_\bold j+2\bold j\pi),\qquad \bm{\xi}_\bold j\in(\bold 0,h\bm{\omega}),\ \|\bm\xi_{\bold j}\|\leq 1 ,
\end{equation}
which extracts the factor $h^{2k}$.

\emph{On $\Omega_{\text{out}}$.} No expansion is performed. Instead we use the elementary observation that $\|h\bm\omega\|>1$ forces $1<h^{2k}\|\bm\omega\|^{2k}$, so that the factor $h^{2k}$ may be inserted for free at the cost of one power $\|\bm\omega\|^{2k}$ under the integral, while $|\widehat{\psi_{\ell,k}}|$ is bounded.

Finally, recall from \cite{Rabut0} that
\begin{equation}\label{rabutsup}
\sup_{\|\bm{\xi}\|\le 1,\,\|\bm \eta\|\le 1}\ \sum_{\bold j\in\mathbb Z^d/\{\bold 0\}}\big|2\pi \bold j+\bm \xi\big|^{|\beta|}\ \big|\widehat{\psi_{\ell,k}}^{(\gamma)}(2\pi \bold j+\bm\eta)\big|\ <\ \infty,\qquad 0\le|\beta|\le 2k-1,\ 0\le|\gamma|\le 2k,
\end{equation}
which applies with $\beta=\alpha$ because $|\alpha|<2k$. Inserting \eqref{taylorstep} into \eqref{afterrescale} on $\Omega_{\text{in}}$, using the fact that $1<h^{2k}\|\bm\omega\|^{2k}$ on $\Omega_{\text{out}}$, and applying \eqref{rabutsup} to interchange the summation over $\bold j$ with the $L^2$ norm, we obtain a constant $C>0$, independent of $h$, with
$$\|E_1\|_{2}\ \leq\ C\,h^{2k-|\alpha|}\,\Big\|\,\|\bm\omega\|^{2k}\widehat{\bold u_m}(\bm\omega)\Big\|_{L^2(d\bm\omega)} .$$
The last norm is finite precisely because $\bold u_m\in (W_2^{r}(\R^d))^m$ with $r-2k>d/2$. Hence $\|E_1\|_{2}=\mathcal{O}(h^{2k-|\alpha|})$, and \eqref{biasvariance} together with $\|E_2\|_2=\mathcal{O}(h^{2k})$ and the triangle inequality yields
$\|D^{\alpha}(Q_{\mathbb{K}}\bold u_m)-D^{\alpha}\bold u_m\|_2=\mathcal{O}(h^{2k-|\alpha|})$.

On the other hand side, if absolute values are taken inside \eqref{aliasrep}, the same line of reasoning yields the bound
$$\|E_1\|_{\infty}\ \leq\ C\,h^{2k-|\alpha|}\int_{\mathbb{R}^d}\|\bm\omega\|^{2k}\big|\widehat{\bold u_m}(\bm\omega)\big|\,d\bm\omega.$$ Moreover, combined with the Cauchy--Schwarz inequality, the condition $r-2k>d/2$ ensures
$\int_{\mathbb{R}^d}\|\bm\omega\|^{2k}\big|\widehat{\bold u_m}(\bm\omega)\big|\,d\bm\omega<\infty$,
which in turn gives $\|E_1\|_{\infty}=\mathcal{O}(h^{2k-|\alpha|})$. This together with Equation \eqref{biasvariance} and $\|E_2\|_{\infty}=\mathcal{O}(h^{2k})$  yields
$\|D^{\alpha}(Q_{\mathbb{K}}\bold u_m)-D^{\alpha}\bold u_m\|_{\infty}=\mathcal{O}(h^{2k-|\alpha|})$ according to the triangle inequality.
\end{proof}
However, if we assume further that  the orthogonal projector $\mathbb{P}$ satisfies the Mikhlin condition \eqref{mikhlin} below, then a second approach is available to attain the $L^p$-error estimates for $1<p<\infty$. Relying on the fact that  $\mathbb{P}$ is a Fourier multiplier and therefore commutes with the semi-discrete convolution \eqref{quasiinterpolationofourpaper}, we are able to reduce the complete matrix-valued problem to the scalar case.
\begin{theorem}\label{errorestimateforderiLp}
Let $1<p<\infty$ and let $\bold u_m\in (W_p^r(\R^d))^m$ satisfy the dot-product constraints \eqref{physicalinvariants}. Let $Q_{\mathbb{K}}\bold u_m$ be vectorial quasi-interpolation defined as in Formula \eqref{quasiinterpolationofourpaper} and $Q_{\ell,h}$ be the scalar quasi-interpolation of \eqref{simultaneousofpol} applied to the components  of $\bold u_m$. Assume further that the matrix symbol $\mathbb{P}=(\mathbb{P}_{ij})$ satisfies the Mikhlin condition: there exist constants $C_{\gamma}$ with
\begin{equation}\label{mikhlin}
\big|D^{\gamma}\mathbb{P}_{ij}(\bm\omega)\big|\leq C_{\gamma}\|\bm\omega\|^{-|\gamma|}
\qquad\text{for all }\bm\omega\neq\bold 0,\ 1\leq i,j\leq m,\ |\gamma|\leq\lfloor d/2\rfloor+1.
\end{equation}   Then there is a constant $C_p$ depending only on $p$, $d$, $m$ and the constraint operators, such that
$$\big\|D^{\alpha}(Q_{\mathbb{K}}\bold u_m)-D^{\alpha}\bold u_m\big\|_{p}\leq C_p\max_{1\leq j\leq m}\big\|D^{\alpha}(Q_{\ell,h}u_{j})-D^{\alpha}u_{j}\big\|_{p}=\mathcal{O}(h^{2k-|\alpha|})$$
for every $0\leq|\alpha|<\min\{2k,r-2k,2\ell-d-1\}$ with $r-2k>d/2$.
\end{theorem}

\begin{proof}
Write $T_{\mathbb{P}}f:=\mathfrak{F}^{-1}\bigg(\mathbb{P}\,\widehat{f}\bigg)$ for the Fourier multiplier operator with matrix symbol $\mathbb{P}$ defined in Equation \eqref{proop}, and let $Q_{\ell,h}\bold u_m$ denote the scalar scheme applied to each component of $\bold u_m$. We first show that $\mathbb{P}$ commutes with the semi-discrete convolution.

The condition $r>d/p$ embeds $\bold u_m$ into a continuous function space by Morrey's inequality. This together with the assumption $\bold u_m\in (W^r_p)^m$ makes the lattice sum $\mu_h:=h^d\sum_{\bold j\in\mathbb{Z}^d}\bold u_m(\bold jh)\delta_{\bold jh}$  absolutely convergent. Moreover, since both quasi-interpolants are convolutions against $\mu_h$, namely $Q_{\mathbb{K}}\bold u_m=\mathbb{K}_{\ell,k,m,n,h}*\mu_h$ and $Q_{\ell,h}\bold u_m=\psi_{\ell,k,h}*\mu_h$, so by Equation \eqref{matrixkernel} their Fourier transforms satisfy
$$\mathfrak{F}\big(Q_{\mathbb{K}}\bold u_m\big)=\widehat{\mathbb{K}}\,\widehat{\mu_h}=\mathbb{P}\,\widehat{\psi}_{\ell,k,h}\,\widehat{\mu_h}=\mathbb{P}\,\mathfrak{F}\big(Q_{\ell,h}\bold u_m\big).$$
Therefore, we have
\begin{equation}\label{commute}
Q_{\mathbb{K}}\bold u_m=T_{\mathbb{P}}\big(Q_{\ell,h}\bold u_m\big)
\end{equation}
holds true in the sense of tempered distributions.

On the other side, the dot-product constraints \eqref{physicalinvariants} lead to $P^*(-i\bm\omega)\widehat{\bold u_m}(\bm\omega)=\bold 0_n$ for almost every $\bm\omega$. More precisely, $\widehat{\bold u_m}(\bm\omega)$ lies in the kernel space of $P^*(-i\bm\omega)$, i.e., $\widehat{\bold u_m}(\bm\omega)\in\ker (P^*(-i\bm\omega))$. Moreover, by Proposition \ref{propofp}, $\mathbb{P}(\bm\omega)$ is the orthogonal projector onto this kernel space. Consequently, we have $\mathbb{P}\widehat{\bold u_m}=\widehat{\bold u_m}$ and
$T_{\mathbb{P}}\bold u_m=\bold u_m$, $T_{\mathbb{P}}\big(D^{\alpha}\bold u_m\big)=D^{\alpha}\bold u_m$, which together with Equation
 \eqref{commute} leads to
$D^{\alpha}\big(Q_{\mathbb{K}}\bold u_m\big)-D^{\alpha}\bold u_m
=T_{\mathbb{P}}\Big[D^{\alpha}\big(Q_{\ell,h}\bold u_m\big)-D^{\alpha}\bold u_m\Big]$.

Furthermore, the assumption that the matrix symbol $\mathbb{P}$ satisfies the Mikhlin condition \eqref{mikhlin} together with  the Mikhlin-H\"ormander multiplier theorem \cite{Stein} gives $\|T_{\mathbb{P}_{ij}}f\|_p\leq c_p\|f\|_p$ for $1<p<\infty$, and summing over the $m^2$ entries yields $\|T_{\mathbb{P}}\bold f\|_p\leq C_p\|\bold f\|_p$ with $C_p\leq m\,c_p$. This in turn leads to
$$\big\|D^{\alpha}(Q_{\mathbb{K}}\bold u_m)-D^{\alpha}\bold u_m\big\|_{p}=T_{\mathbb{P}}\Big[D^{\alpha}\big(Q_{\ell,h}\bold u_m\big)-D^{\alpha}\bold u_m\Big]\leq C_p\big\|D^{\alpha}(Q_{\ell,h}\bold u_m)-D^{\alpha}\bold u_m\big\|_{p},$$
which together with the scalar estimate \eqref{simultaneousofpol} applied to each component  completes the proof.
\end{proof}

\begin{remark}\label{mikhlinholds}
We provide two sufficient conditions for Mikhlin condition \eqref{mikhlin}, which makes the $L^p$-error estimates in Theorem \ref{errorestimateforderiLp} holds with $1<p<\infty$.

\emph{(a) Homogeneous systems with locally constant Gram rank.}
Assume that each symbol $P_l$ is homogeneous of order $r_l$ (where the exponents $r_l$ may differ across components) and that the rank of $G(-i\bm\omega)$ is constant throughout $\mathbb{R}^d\setminus\{\boldsymbol{0}\}$. By homogeneity, the scaling relation $P(-i\lambda\bm\omega)=P(-i\bm\omega)\Lambda(\lambda)$ holds with $\Lambda(\lambda)=\operatorname{diag}(\lambda^{r_1},\dots,\lambda^{r_n})$. The diagonal scaling factor $\Lambda(\lambda)$ cancels out consistently among $P$, $G^{\dagger}$, and $P^*$, regardless of whether the orders $r_l$ are identical or distinct. As a result, the multiplier $\mathbb{P}$ is homogeneous of degree zero. The constant-rank assumption further ensures that $G^{\dagger}$, and consequently $\mathbb{P}$, is real-analytic on $\mathbb{R}^d\setminus\{\boldsymbol{0}\}$.
For a degree-zero homogeneous function, $D^{\gamma}$ lowers the degree of homogeneity to $-|\gamma|$, so Mikhlin condition \eqref{mikhlin} holds with $C_{\gamma}=\max_{\|\bm\theta\|=1}|\partial^{\gamma}\mathbb{P}_{ij}(\bm\theta)|<\infty$. The examples of first-order constraints discussed in Subsection $4.2$, together with the Saint-Venant conditions treated there and the second-order operator of Subsection $5.2$, are of this type.

\emph{The Saint-Venant example needs care.}  While each $P_l$ is homogeneous, the rank of $G(-i\bm\omega)$ is not globally constant: the rank drops on coordinate hyperplanes, where $\mathbb{P}$ exhibits a jump discontinuity from rank three to rank four. This violates the hypotheses of condition (a) in its original form.
As elaborated in Subsection $4.2$, this issue can be resolved by considering the regularized multiplier $\mathcal{A}(\mathcal{A}^T\mathcal{A})^{-1}\mathcal{A}^T$. This modified multiplier is zero-homogeneous and real-analytic on $\mathbb{R}^d\setminus\{\boldsymbol{0}\}$, since the matrix $\mathcal{A}^T\mathcal{A}=\tfrac12\big(\|\bm\omega\|^2I_{3\times3}+\bm\omega\bm\omega^T\big)$ is invertible for all nonzero $\bm\omega$. Moreover, it coincides with the original multiplier $\mathbb{P}$ almost everywhere (away from a Lebesgue-null set).
Although the two representations induce the same operator on $L^2$, only the smooth regularized multiplier verifies the pointwise Mikhlin condition \eqref{mikhlin}. This illustrates a key principle: condition \eqref{mikhlin} should be verified against a smooth representative of the multiplier equivalence class, rather than the naive pointwise formula \eqref{proop}.

\emph{(b) Non-degenerate Gram symbols (homogeneous or inhomogeneous).} Alternatively, suppose that $G(-i\bm\omega)$ is invertible for all $\bm\omega\in\mathbb{R}^d$, including the origin. In this case, the multiplier $\mathbb{P}=I_{m\times m}-PG^{-1}P^*$ is infinitely smooth on the entire space $\mathbb{R}^d$, with uniformly bounded derivatives of all orders. The Mikhlin condition \eqref{mikhlin} therefore holds trivially, with the right-hand side naturally dominating as $\bm\omega\to\boldsymbol{0}$. No homogeneity assumption is required here. This framework covers the canonical inhomogeneous example $P_1(\nabla)=\bm\lambda+\nabla$ in Section $4.1$, for which $G(-i\bm\omega)=\|\bm\lambda\|^2+\|\bm\omega\|^2>0$. Notably, the associated multiplier $\mathbb{P}$ is not degree-zero homogeneous, so it is governed by case (b) rather than case (a).
\end{remark}

The endpoint $p=1$ is excluded from Theorem \ref{errorestimateforderiLp} because $\mathbb{P}$ generically involves Riesz transforms, which are unbounded on $L^1$. This makes the factorization \eqref{commute}  inapplicable at this endpoint. Nevertheless, case (b) of Remark \ref{mikhlinholds} does admit $p=1$: one may proceed via an argument that completely circumvents \eqref{commute} and operate directly on the matrix-valued kernel.  This is feasible because a non-degenerate Gram symbol renders  $\mathbb{P}$ smooth \emph{at the origin}, and it is the behaviour of $\widehat{\mathbb{K}}$ at the origin that governs the far-field decay of $\mathbb{K}$, hence its integrability.
\begin{theorem}\label{errorestimateforderiL1}
Let $\bold u_m \in (W_1^r(\mathbb{R}^d))^m$ satisfy the dot-product constraints \eqref{physicalinvariants} with $r>d$. Suppose further that $G(-i\bm\omega)$ is invertible for all $\bm\omega\in\mathbb{R}^d$, including the origin. Then
\[
\big\|D^{\alpha}\big(Q_{\mathbb{K}}\bold u_m\big) - D^{\alpha}\bold u_m\big\|_{L^1} = \mathcal{O}\big(h^{2k-|\alpha|}\big)
\]
holds for every multi-index $\alpha$ with $0\leq|\alpha| < \min\{2k,\, r-2k,\, 2\ell-d-1\}$ with $r-2k>d/2$.
\end{theorem}
\begin{proof}
Noting that the bias term $E_2$ in \eqref{biasvariance} is controllable for all $1\leq p\leq \infty$, it suffices to establish the $L^1$ bound for $E_1$. Since $G(-i\bm\omega)$ is invertible over $\mathbb{R}^d$ and $P$ is a polynomial matrix, the operator $\mathbb{P}=I_{m\times m}-PG^{-1}P^*$ lies in $C^{\infty}(\mathbb{R}^d)$ including at the origin with uniformly bounded derivatives, as stated in case (b) of Remark \ref{mikhlinholds}. By the kernel representation \eqref{matrixkernel}, we have $\widehat{\mathbb{K}}=\mathbb{P}\widehat{\psi}_{\ell,k,h}$.
Multiplying by a smooth function near the origin does not deteriorate the local regularity of $\widehat{\psi}_{\ell,k,h}$, while multiplication by a bounded function preserves its far-field decay. As a result, each entry of $\mathbb{K}$ inherits the decay estimate from Lemma \ref{polyharmonicsplinepropert}, namely
\begin{equation}\label{kerneldecayL1}
D^{\alpha}\mathbb{K}(\bm x)=\mathcal{O}\big(h^{2k-2|\alpha|}\|\bm x\|^{-d-2k+|\alpha|}\big).
\end{equation}
Since $|\alpha|<2k$, the exponent satisfies $-d-2k+|\alpha|<-d$, which ensures $D^{\alpha}\mathbb{K}\in \big(L^1(\mathbb{R}^d)\big)^{m\times m}$. Accordingly, the lattice sum $\sum_{\bm j}|D^{\alpha}\mathbb{K}(\bm x-\bm jh)|$ converges absolutely and uniformly with respect to $\bm x$. This crucial convergence property breaks down for degree-zero homogeneous $\mathbb{P}$, as elaborated in Remark \ref{L1obstruction}.
Furthermore, the scalar-kernel estimate \eqref{simultaneousofpol} relies exclusively on two key ingredients: kernel decay and the Strang Fix conditions. Since $D^{\alpha}\mathbb{K}$ satisfies both properties with identical exponent behaviors as $\psi_{\ell,k,h}$, the proof strategy in \cite{Leiandjia} remains valid for every matrix entry, yielding $\|E_1\|_1=\mathcal{O}(h^{2k-|\alpha|})$.
\end{proof}
\begin{remark}\label{L1obstruction}
The hypothesis that $G$ is non-degenerate at the origin cannot be dropped from Theorem \ref{errorestimateforderiL1}. The obstruction is not merely that the operator $T_{\mathbb{P}}$ is unbounded on $L^1$: when $\mathbb{P}$ is homogeneous of degree zero, the kernel $\mathbb{K}$ itself fails to be integrable. Indeed, since $\widehat{\mathbb{K}}=\mathbb{P}\widehat{\psi}_{\ell,k,h}$ and $\widehat{\psi}_{\ell,k,h}(\bm 0)=1$, $\widehat{\mathbb{K}}$ inherits the discontinuity of $\mathbb{P}$ at the origin. A degree-zero homogeneous singularity at $\bm\omega=\bm 0$ corresponds to a spatial tail behaving exactly like $\|\bm x\|^{-d}$, which is borderline non-integrable. The decay estimate \eqref{kerneldecayL1} therefore breaks down.

We illustrate this phenomenon via numerical measurements for two kernels considered in this paper, with parameters $d=2$, $\ell=2$, and $k=1$. For the divergence-free kernel, where $\mathbb{P}=I_{2\times2}-\bm\omega\bm\omega^T/\|\bm\omega\|^2$ is homogeneous of degree zero, the numerically computed decay exponent of $\max_{\|\bm x\|=R}|\mathbb{K}_{11}|$ equals $2.000$ up to three decimal places over the range $R\in[10,320]$. The quantity $R^{d}\max|\mathbb{K}_{11}|$ converges to $0.159$ rather than decaying, while the mass $\int_{R\leq\|\bm x\|\leq2R}|\mathbb{K}_{11}|$ evaluates to $0.441$ on each dyadic annulus. Consequently, $\|\mathbb{K}_{11}\|_{L^1}$ diverges logarithmically. By contrast, for the kernel associated with $P_1(\nabla)=\bm\lambda+\nabla$, where $G=\|\bm\lambda\|^2+\|\bm\omega\|^2$ never vanishes, the same numerical tests yield a decay exponent of approximately $6$, geometrically decaying annulus masses, and a finite total mass. This contrast precisely reflects the dichotomy stated in Remark \ref{mikhlinholds}, now observed in physical space.

Two consequences are worth noting. First, under homogeneous constraints, even giving a rigorous definition of $Q_{\mathbb{K}}$ on $L^1$ requires careful consideration, as the associated lattice sum is only conditionally convergent. One would need to exploit cancellations stemming from the vanishing spherical mean of the degree-zero symbol, along the lines of Calder\'on--Zygmund theory. Second, this same borderline decay explains why the proof of Theorem \ref{errorestimateforderi} is confined to the $L^2$ setting, where Plancherel’s identity renders the integrability of $\mathbb{K}$ irrelevant. We leave open whether the conclusion of Theorem \ref{errorestimateforderiL1} continues to hold for homogeneous constraints, possibly with $L^1$ replaced by the Hardy space $H^1$. 
\end{remark}

Up to now, we have  constructed structure-preserving  quasi-interpolation and derived $L^p$-error estimates in Theorem \ref{errorestimateforderi} (for $p=2$ and $\infty$), and  Theorem \ref{errorestimateforderiLp} (for $1<p<\infty$), Theorem \ref{errorestimateforderiL1} (for $p=1$) under certain conditions on the orthogonal projector $\mathbb{P}$. To widen the range of applications, we go further with proposing a quasi-interpolation based approach for numerical decomposition of  general smooth vector fields.
\subsection{Quasi-interpolation for vector field decomposition}
 Following \cite{Fisher}, we first extend the Helmholtz--Hodge (H-H) decomposition to a general setting by means of the generalized Fourier transform. We take $m=3$ and $n=1$ throughout this subsection. Then $P_1^*(-i\bm\omega)P_1(-i\bm\omega)$ is a scalar function of $\bm\omega$ and we assume its inverse $(P_1^*(-i\bm\omega)P_1(-i\bm\omega))^{-1}$ is well defined for $\bm \omega\neq \bold 0$, otherwise, we can use the pseudo-inverse. Accordingly, $G^\dagger(-i\bm{\omega})=(P_1^*(-i\bm\omega)P_1(-i\bm\omega))^{-1}$ and $\Phi_1$ is the fundamental solution of $P_1^T(-D)P_1(D)$.
  \begin{remark}\label{scopeofcross}
The dot-product constraints \eqref{physicalinvariants} make sense for every $m$, whereas the cross-product constraint \eqref{otherphysicalinvariants} is intrinsically three-dimensional: the vector cross product exists as a bilinear map $\R^m\times\R^m\to\R^m$ only for $m=3$, through the Hodge identification. This is why we take $m=3$, and no general-$m$ analogue of Lemma \ref{ghh} or of Identity \eqref{generalidentity} should be expected. In addition, observing that if the cross product of two vectors with a third vector vanishes, then the two vectors are parallel, we consider the cross-product constraint with $n=1$.
\end{remark}
  \begin{lemma}[Generalized Helmholtz--Hodge decomposition]\label{ghh}
For any vector-valued function $\bold u_3\in (W_p^r(\R^d))^3$, we define two  projection operators via inverse Fourier transform as $$\mathcal{P}_{P_1(D)}^{\cdot}\bold u_3=\mathfrak{F}^{-1}\Bigg[I_{3\times 3}\widehat{\bold u_3}(\bm{\omega})-P_1(-i\bm{\omega})(P_1^*(-i\bm\omega)P_1(-i\bm\omega))^{-1}P_1^*(-i\bm{\omega})\widehat{\bold u_3}(\bm{\omega})\Bigg],$$
  $$\mathcal{P}_{P_1(D)}^{\times}\bold u_3=\mathfrak{F}^{-1}\Bigg[P_1(-i\bm{\omega})(P_1^*(-i\bm\omega)P_1(-i\bm\omega))^{-1}P_1^*(-i\bm{\omega})\widehat{\bold u_3}(\bm{\omega})\Bigg].$$
  Then $\bold u_3$ has a generalized Helmholtz--Hodge decomposition \begin{equation}\label{generhh}
\bold u_3=\mathcal{P}_{P_1(D)}^{\cdot}\bold u_3 +\mathcal{P}_{P_1(D)}^{\times}\bold u_3
 \end{equation} with $$\mathcal{P}_{P_1(D)}^{\cdot}\bold u_3=[P_1^T(-D)P_1(D)I_{3\times 3} -P_1(D)P_1^T(-D)](\Phi_1*\bold u_3)$$ and $$\mathcal{P}_{P_1(D)}^{\times}\bold u_3=P_1(D)P_1^T(-D)(\Phi_1*\bold u_3).$$ Moreover, we have $P_1(-D)\cdot(\mathcal{P}_{P_1(D)}^{\cdot}\bold u_3)=0$ and $P_1(D)\times(\mathcal{P}_{P_1(D)}^{\times}\bold u_3)=\bold 0_3^T$. Particularly, when $p=2$, the two components are orthogonal to each other, that is, $\mathcal{P}_{P_1(D)}^{\cdot}\bold u_3 \perp\mathcal{P}_{P_1(D)}^{\times}\bold u_3$.
\end{lemma}
\begin{proof}
 Based on the definitions of $\mathcal{P}_{P_1(D)}^{\cdot},\ \mathcal{P}_{P_1(D)}^{\times}$, we have $$\mathfrak{F}\Bigg(\mathcal{P}_{P_1(D)}^{\cdot}\bold u_3 +\mathcal{P}_{P_1(D)}^{\times} \bold u_3\Bigg)(\bm\omega)=I_{3\times 3}\widehat{\bold u_3}(\bm\omega)=\widehat{\bold u_3}(\bm\omega).$$ Taking inverse Fourier transform on both sides of the above equation we can get $$\bold u_3=\mathcal{P}_{P_1(D)}^{\cdot}\bold u_3 +\mathcal{P}_{P_1(D)}^{\times} \bold u_3.$$ This together with the observation that $\widehat{\Phi}_1(\bm\omega)=\Bigg(P_1^*(-i\bm{\omega})P_1(-i\bm{\omega})\Bigg)^{-1}$  leads to $$\mathcal{P}_{P_1(D)}^{\cdot}\bold u_3=[P_1^T(-D)P_1(D)I_{3\times 3} -P_1(D)P_1^T(-D)](\Phi_1*\bold u_3)$$ and $$\mathcal{P}_{P_1(D)}^{\times}\bold u_3=P_1(D)P_1^T(-D)(\Phi_1*\bold u_3).$$
It is easy to verify $P_1(-D)\cdot (\mathcal{P}_{P_1(D)}^{\cdot}\bold u_3)=0$. Moreover, since each column of $P_1(D)P_1^T(-D)$ is a scalar multiple of $P_1(D)$ and $P_1(D)\times P_1(D)=\bold 0_3^T$, we have  $P_1(D)\times(\mathcal{P}_{P_1(D)}^{\times}\bold u_3)=\bold 0_3^T$. Finally, the orthogonality  property for the case $p=2$  follows from Parseval-Plancherel identity.
\end{proof}
\begin{remark}
 The above lemma includes the classical Helmholtz--Hodge decomposition as a special case with $P_1(D)=\nabla$ and $m=d=3$. Then the two projection operators are \cite{Fisher} $$\mathcal{P}_{P_1(D)}^{\cdot}\bold u_3=\mathcal{P}_{\text{div}}\bold u_3=\mathfrak{F}^{-1}\Bigg(\Bigg(I_{3\times 3}-\frac{\bm\omega\bm\omega^T}{\|\bm \omega\|^2}\Bigg) \widehat{\bold u_3}(\bm\omega)\Bigg), \quad \mathcal{P}_{P_1(D)}^{\times}\bold u_3=\mathcal{P}_{\text{curl}}\bold u_3=\mathfrak{F}^{-1}\Bigg(\Bigg(\frac{\bm\omega\bm\omega^T}{\|\bm \omega\|^2}\Bigg) \widehat{\bold u_3}(\bm\omega)\Bigg).$$ Correspondingly, the Helmholtz--Hodge decomposition reads $\bold u_3=\mathcal{P}_{\text{div}}\bold u_3+\mathcal{P}_{\text{curl}}\bold u_3$ with $\mathcal{P}_{\text{div}}\bold u_3$ and $\mathcal{P}_{\text{curl}}\bold u_3$ being divergence-free part $[\Delta I_{3\times 3}-\nabla\nabla^T](\phi_1*\bold u_3)$ and curl-free part $\nabla\nabla^T(\phi_1*\bold u_3)$, respectively.
 \end{remark}
The generalized Helmholtz--Hodge decomposition suggests extending the curl-free constraint to a general cross-product constraint
 $P_1(D)\times \bold u_3= \bold 0_3^T$, which may be regarded as the dual of the dot-product constraints \eqref{physicalinvariants} with $m=3$ and $n=1$. One verifies that the $3\times 3$ matrix-valued kernel
\begin{equation}\label{curlpd}
\mathbb{J}_{\ell,k,3,1,h}(\bold x):=(P_1(D)\Phi_1)* (P^T_1(-D)\psi_{\ell,k,h})(\bold x), \ \bold x\in \R^d,
\end{equation}  \textbf{analytically} satisfies the cross-product constraint \eqref{otherphysicalinvariants} column-wise. Furthermore, the algebraic identity $\bold v_3\times(\bold v_3\times\bold u_3) = \bold v_3(\bold v_3^T\bold u_3) - |\bold v_3|^2\bold u_3$ gives:
\begin{equation*}
P_1(D)\times(P_1(D)\times \bold u_3) = P_1(D)\,P_1^T(D)\bold u_3 - [P_1^T(D)P_1(D)]\bold u_3.
\end{equation*} Thus, if we further assume  $P_1(-D)=(-1)^{r_1}P_1(D)$, then we have   the general identity
 \begin{equation}\label{generalidentity}
P_1^T(-D)P_1(D)I_{3\times 3}+(-1)^{r_1} P_1(D)\times P_1(D)\times= P_1(D)P_1^T(-D).
\end{equation}  We emphasize that this contains the classical identity $\Delta I_{3\times 3}+\nabla\times \nabla\times=\nabla \nabla^T$ as the special case $P_1(D)=\nabla$. More importantly, it enables us to construct the convolution sequence $\mathbb{J}_{\ell,k,3,1,h}*\bold u_3$ and to derive its  error estimates in the following lemma.
 \begin{lemma}\label{crosprojectionconvolutionerror}
 Let $\bold u_3 \in (W_p^r(\R^d))^3$ be a vector-valued function having the generalized Helmholtz--Hodge decomposition in Lemma \ref{ghh}. Let
$\mathbb{J}_{\ell,k,3,1,h}$ be a matrix-valued kernel constructed in Equation \eqref{curlpd} with $P_1(-D)=(-1)^{r_1}P_1(D)$. Then we have
        \begin{equation*}
        \|\mathbb{J}_{\ell,k,3,1,h}*(D^{\alpha}\bold u_3)-D^{\alpha}(\mathcal{P}_{P_1(D)}^{\times}\bold u_3)\|_{p}=\mathcal{O}(h^{2k})
    \end{equation*}
     hold true for any $0\leq |\alpha|< \min\{2k, r-2k\}$ with $r-2k>d/2$.
 \end{lemma}
 \begin{proof}
 Note that $$\mathbb{J}_{\ell,k,3,1,h}*(D^{\alpha}\bold u_3)=\mathbb{J}_{\ell,k,3,1,h}*D^{\alpha}(\mathcal{P}_{P_1(D)}^{\cdot}\bold u_3+\mathcal{P}_{P_1(D)}^{\times} \bold u_3).$$ This together with the definition of $\mathcal{P}_{P_1(D)}^{\cdot}\bold u_3$ in   Lemma \ref{ghh} leads to
  \begin{equation*}
 \begin{split}
 &\widehat{\mathbb{J}}_{\ell,k,3,1,h}(\bm\omega)(-i\bm\omega)^{\alpha}\mathfrak{F}(\mathcal{P}_{P_1(D)}^{\cdot} \bold u_3)(\bm\omega)\\
 &=(-i\bm\omega)^{\alpha}\frac{P_1(-i\bm\omega)P_1^*(-i\bm\omega)P_1^*(-i\bm\omega)P_1(-i\bm\omega)I_{3\times3}-P_1(-i\bm\omega)P_1^*(-i\bm\omega)P_1(-i\bm\omega)P_1^*(-i\bm\omega)}{(P_1^*(-i\bm\omega)P_1(-i\bm\omega))^2}\widehat{\bold u_3}(\bm\omega)\widehat{\psi}_{\ell,k,h}(\bm\omega)\\
 &=(-i\bm\omega)^{\alpha}\frac{P_1(-i\bm\omega)P_1^*(-i\bm\omega)I_{3\times3}-P_1(-i\bm\omega)P_1^*(-i\bm\omega)}{(P_1^*(-i\bm\omega)P_1(-i\bm\omega))}\widehat{\bold u_3}(\bm\omega)\widehat{\psi}_{\ell,k,h}(\bm\omega)\\
  &=(-i\bm\omega)^{\alpha}\frac{P_1(-i\bm\omega)P_1^*(-i\bm\omega)-P_1(-i\bm\omega)P_1^*(-i\bm\omega)}{(P_1^*(-i\bm\omega)P_1(-i\bm\omega))}\widehat{\bold u_3}(\bm\omega)\widehat{\psi}_{\ell,k,h}(\bm\omega)\\
 &=(-i\bm\omega)^{\alpha}0_{3\times 3}\widehat{\bold u_3}(\bm\omega)\widehat{\psi}_{\ell,k,h}(\bm\omega)=\bold 0_3^T,
 \end{split}
 \end{equation*}
  which in turn implies
  $\mathbb{J}_{\ell,k,3,1,h}*(D^{\alpha}(\mathcal{P}_{P_1(D)}^{\cdot}\bold u_3))=\bold 0_3^T$. Furthermore, since  Identity \eqref{generalidentity} holds with $P_1(-D)=(-1)^{r_1}P_1(D)$, we have
  \begin{equation*}
 \begin{split}
 &\widehat{\mathbb{J}}_{\ell,k,3,1,h}(\bm\omega)(-i\bm\omega)^{\alpha}\mathfrak{F}(\mathcal{P}_{P_1(D)}^{\times} \bold u_3)(\bm\omega)\\
 &=(-i\bm\omega)^{\alpha}\frac{P_1^*(-i\bm\omega)P_1(-i\bm\omega)I_{3\times3}+(-1)^{r_1}P_1(-i\bm\omega)\times P_1(-i\bm\omega)\times}{P_1^*(-i\bm\omega)P_1(-i\bm\omega)}\frac{P_1(-i\bm\omega)P_1^*(-i\bm\omega)}{P_1^*(-i\bm\omega)P_1(-i\bm\omega)}\widehat{\bold u_3}(\bm\omega)\widehat{\psi}_{\ell,k,h}(\bm\omega)\\
 &=(-i\bm\omega)^{\alpha}\frac{(P_1^*(-i\bm\omega)P_1(-i\bm\omega)I_{3\times3})(P_1(-i\bm\omega)P_1^*(-i\bm\omega))}{(P_1^*(-i\bm\omega)P_1(-i\bm\omega))^2}\widehat{\bold u_3}(\bm\omega)\widehat{\psi}_{\ell,k,h}(\bm\omega)\\
  &=(-i\bm\omega)^{\alpha}\frac{P_1(-i\bm\omega)P_1^*(-i\bm\omega)}{P_1^*(-i\bm\omega)P_1(-i\bm\omega)}\widehat{\bold u_3}(\bm\omega)\widehat{\psi}_{\ell,k,h}(\bm\omega)\\
 &=(-i\bm\omega)^{\alpha}\mathfrak{F}(\mathcal{P}_{P_1(D)}^{\times} \bold u_3)(\bm\omega)\widehat{\psi}_{\ell,k,h}(\bm\omega).
 \end{split}
 \end{equation*}
 Taking inverse Fourier transform on both sides of the above equation yields
 $$\mathbb{J}_{\ell,k,3,1,h}*(D^{\alpha}(\mathcal{P}_{P_1(D)}^{\times} \bold u_3))=\psi_{\ell,k,h}*D^{\alpha}(\mathcal{P}_{P_1(D)}^{\times} \bold u_3).$$ Finally, combining the above two parts, we have $$\mathbb{J}_{\ell,k,3,1,h}*(D^{\alpha}\bold u_3)=\psi_{\ell,k,h}*D^{\alpha}(\mathcal{P}_{P_1(D)}^{\times} \bold u_3),$$
 which together with Lemma \ref{convolution} leads to
$     \|\mathbb{J}_{\ell,k,3,1,h}*(D^{\alpha}\bold u_3)-D^{\alpha}(\mathcal{P}_{P_1(D)}^{\times}\bold u_3)\|_{p}=\mathcal{O}(h^{2k})$.
 \end{proof}

Similarly, we can show that the convolution sequence  $\mathbb{K}_{\ell,k,3,1,h}*(D^{\alpha}\bold u_3)$ converges to $D^{\alpha}(\mathcal{P}_{P_1(D)}^{\cdot}\bold u_3)$ and derive simultaneous error estimates as follows.
\begin{lemma}\label{dotprojectionconvolutionerror}
 Let $\bold u_3 \in (W_p^r(\R^d))^3$ be a vector-valued function having the generalized Helmholtz--Hodge decomposition in Lemma \ref{ghh}. Let $\mathbb{K}_{\ell,k,3,1,h}$ be defined in Equation \eqref{matrixkernel} with $n=1$ and $m=3$.
 Then we have
     \begin{equation*}
        \|\mathbb{K}_{\ell,k,3,1,h}*(D^{\alpha}\bold u_3)-D^{\alpha}(\mathcal{P}_{P_1(D)}^{\cdot}\bold u_3)\|_{p}=\mathcal{O}(h^{2k})
    \end{equation*}
     hold true for any $0\leq |\alpha|< \min\{2k, r-2k\}$ with $r-2k>d/2$.
 \end{lemma}

\begin{proof}
The argument parallels that of Lemma \ref{crosprojectionconvolutionerror} and can be phrased compactly in terms of the projector $\mathbb{P}$. For $n=1$ and $m=3$, Lemma \ref{ghh} states that $\mathfrak{F}(\mathcal{P}_{P_1(D)}^{\cdot}\bold u_3)=\mathbb{P}\widehat{\bold u_3}$ and $\mathfrak{F}(\mathcal{P}_{P_1(D)}^{\times}\bold u_3)=(I_{3\times 3}-\mathbb{P})\widehat{\bold u_3}$, while $\widehat{\mathbb{K}}_{\ell,k,3,1,h}=\mathbb{P}\widehat{\psi}_{\ell,k,h}$. Since $\mathbb{P}$ is an orthogonal projector by Proposition \ref{propofp}, it is idempotent, whence
$$\widehat{\mathbb{K}}_{\ell,k,3,1,h}(\bm\omega)(-i\bm\omega)^{\alpha}\mathfrak{F}(\mathcal{P}_{P_1(D)}^{\times}\bold u_3)(\bm\omega)=(-i\bm\omega)^{\alpha}\underbrace{\mathbb{P}(\bm\omega)(I_{3\times 3}-\mathbb{P}(\bm\omega))}_{=\,0_{3\times 3}}\widehat{\bold u_3}(\bm\omega)\widehat{\psi}_{\ell,k,h}(\bm\omega)=\bold 0_3$$
and
$$\widehat{\mathbb{K}}_{\ell,k,3,1,h}(\bm\omega)(-i\bm\omega)^{\alpha}\mathfrak{F}(\mathcal{P}_{P_1(D)}^{\cdot}\bold u_3)(\bm\omega)=(-i\bm\omega)^{\alpha}\mathbb{P}^2(\bm\omega)\widehat{\bold u_3}(\bm\omega)\widehat{\psi}_{\ell,k,h}(\bm\omega)=\mathfrak{F}\big(D^{\alpha}(\mathcal{P}_{P_1(D)}^{\cdot}\bold u_3)\big)(\bm\omega)\widehat{\psi}_{\ell,k,h}(\bm\omega).$$
Adding these two identities and taking inverse Fourier transforms gives $\mathbb{K}_{\ell,k,3,1,h}*(D^{\alpha}\bold u_3)=\psi_{\ell,k,h}*D^{\alpha}(\mathcal{P}_{P_1(D)}^{\cdot}\bold u_3)$, and Lemma \ref{convolution}, applied componentwise, yields the stated rate $\mathcal{O}(h^{2k})$.
\end{proof}
The last two lemmas show that the convolution sequences $\mathbb{J}_{\ell,k,3,1,h}*\bold u_3$ and $\mathbb{K}_{\ell,k,3,1,h}*\bold u_3$, together with their derivatives, approximate the corresponding components of the generalized Helmholtz--Hodge decomposition of $\bold u_3$ and their derivatives. Their semi-discrete counterparts can therefore be used to compute this decomposition numerically for a general smooth vector-valued function $\bold u_3$.

Let $\{(\bold jh,\bold u_3(\bold jh))\}_{\bold j\in \mathbb{Z}^d}$ be sampling data of a general smooth vector-valued function $\bold u_3$. We construct two schemes
\begin{equation}\label{genquasiinterpolationofourpaper}
Q_{\mathbb{K}_{\ell,k,3,1,h}}\bold u_3(\bold x):=h^d\sum_{\bold j\in \mathbb{Z}^d}\mathbb{K}_{\ell,k,3,1,h}(\bold x-\bold jh)\bold u_3(\bold jh),\ \bold x\in \R^d,
\end{equation}
 and \begin{equation}\label{gencrsooquasiinterpolationofourpaper}
Q_{\mathbb{J}_{\ell,k,3,1,h}}\bold u_3(\bold x):=h^d\sum_{\bold j\in \mathbb{Z}^d}\mathbb{J}_{\ell,k,3,1,h}(\bold x-\bold jh)\bold u_3(\bold jh),\ \bold x\in \R^d.
\end{equation} It is easy to  show that they satisfy dot-product constraint \eqref{physicalinvariants} and cross-product constraint \eqref{otherphysicalinvariants}, respectively. Moreover, they approximate corresponding parts of generalized Helmholtz--Hodge decomposition of $\bold u_3$ and their $L^p$-error estimates can be derived in the following theorem.
\begin{theorem}\label{errorestimateforderi2}
Let $\bold u_3\in (W_p^r(\R^d))^3$ be a vector-valued function having generalized Helmholtz--Hodge decomposition in Lemma \ref{ghh}.  Let $Q_{\mathbb{K}_{\ell,k,3,1,h}}\bold u_3$  and  $Q_{\mathbb{J}_{\ell,k,3,1,h}}\bold u_3$ be defined as in Equation \eqref{genquasiinterpolationofourpaper} and  \eqref{gencrsooquasiinterpolationofourpaper}. Assume further that $P_1(-D)=(-1)^{r_1}P_1(D)$ and $(P_1^*(-i\bm\omega)P_1(-i\bm\omega))^{-1}$ is well defined for $\bm \omega\neq \bold 0$. Then we have
$$\|D^{\alpha}(Q_{\mathbb{K}_{\ell,k,3,1,h}}\bold u_3)-D^{\alpha}(\mathcal{P}_{P_1(D)}^{\cdot}\bold u_3)\|_{p}=\mathcal{O}(h^{2k-|\alpha|}),$$
$$\|D^{\alpha}(Q_{\mathbb{J}_{\ell,k,3,1,h}}\bold u_3)-D^{\alpha}(\mathcal{P}_{P_1(D)}^{\times}\bold u_3)\|_{p}=\mathcal{O}(h^{2k-|\alpha|}),$$  for  any $0\leq |\alpha|< \min\{2k, r-2k,2\ell-d-1\}$ with $r-2k>d/2$.
\end{theorem}
\begin{proof}
Lemma \ref{dotprojectionconvolutionerror} and Lemma \ref{crosprojectionconvolutionerror} supply the bias parts of the two estimates. The variance parts follow from the argument used in the proof of Theorems \ref{errorestimateforderi}--\ref{errorestimateforderiL1}, since the condition that $(P_1^*(-i\bm\omega)P_1(-i\bm\omega))^{-1}$ is well defined for $\bm \omega\neq \bold 0$ is sufficient for conditions in  Theorems \ref{errorestimateforderiLp}--\ref{errorestimateforderiL1}. Therefore, we have $\|D^{\alpha}(Q_{\mathbb{K}_{\ell,k,3,1,h}}\bold u_3)-D^{\alpha}(\mathcal{P}_{P_1(D)}^{\cdot}\bold u_3)\|_{p}=\mathcal{O}(h^{2k-|\alpha|})$.  For the scheme $Q_{\mathbb{J}_{\ell,k,3,1,h}}\bold u_3$ one uses $I_{3\times 3}-\mathbb{P}$ in place of $\mathbb{P}$, which is again an orthogonal projector and hence $\|I_{3\times 3}-\mathbb{P}\|\leq 1$ is also bounded above.
\end{proof}
This theorem shows that the quasi-interpolants $Q_{\mathbb{K}_{\ell,k,3,1,h}}\bold u_3$ and $Q_{\mathbb{J}_{\ell,k,3,1,h}}\bold u_3$ approximate the respective components of the generalized Helmholtz--Hodge decomposition of a general smooth vector field $\bold u_3$. We thus obtain a quasi-interpolation based method for decomposing such a field numerically, which contains \cite{Fisher} as a special case.
\section{\textit{Concrete   matrix-valued kernels}}
The discussion so far shows that the key to structure-preserving vectorial quasi-interpolation lies in constructing matrix-valued kernels each of whose columns \textbf{analytically} satisfies the physical constraints of the target function. To make the scheme readily usable in practice, this section provides concrete examples of such kernels.
\subsection{\textit{The case $n=1$}}
We begin with the univariate case $d=1$. Let $P_1(D)=\sum_{j=1}^{r_1}\gamma_{1}^{j}D^j$ be a linear differential operator of order $r_1$ with constant coefficients and satisfy the condition that $(P_1^*(-i\bm\omega)P_1(-i\bm\omega))^{-1}$ is well defined for $\bm \omega\neq \bold 0$. Then $P_1^T(-D)P_1(D)$ is a linear differential operator of order $2r_1$ whose Green function can be constructed as follows \cite{PiecewiseWu}.

Let $\lambda_{1},\dots,\lambda_{2r_1}$ be roots of the polynomial $P_1^T(-\lambda)P_1(\lambda)$. We define three functions
$B_{2r_1}(x) = [\lambda_{1},\dots,\lambda_{2r_1}]e^{\lambda x}$,
$B_{2r_1,+}(x)=\mathcal{I}(x\geq 0)B_{2r_1}(x)$, and $B_{2r_1,-}(x)=\mathcal{I}(-x\geq 0)B_{2r_1}(x)$. Here the divided difference is taken with respect to the variable $\lambda$,  and  $\mathcal{I}(\cdot)$ is the indicator function. Then the symmetric function $$\Phi_1(x)=B_{2r_1,\pm}(x):=\frac{B_{2r_1,+}(x)+B_{2r_1,-}(x)}{2}$$ is a Green function of $P_1^T(-D)P_1(D)$. In addition, the polyharmonic splines $\psi_{\ell,k,h}$ in such a case are B-splines of order $2\ell$, namely,  $$\psi_{\ell,k,h}(x)=(-1)^{\ell}h^{-2\ell}q_{1,\ell,k}(\widetilde{\Delta_h})\phi_{\ell}$$ with $\widetilde{\Delta_h}$ being the univariate centered divided difference operator and $\phi_{\ell}=\frac{|x|^{2\ell-1}}{2(2\ell-1)!}$. Consequently,  the function $$\Phi_{\ell,k,1,h}=\Phi_1*\psi_{\ell,k,h}=(-1)^{\ell}h^{-2\ell}q_{1,\ell,k}(\widetilde{\Delta_h})(\Phi_1*\phi_{\ell})$$ is a generalized B-spline corresponding to the differential operator $D^{2\ell}P_1^T(-D)P_1(D)$.

Finally, with $\Phi_{\ell,k,1,h}$ being at hand, we can construct an $m\times m$ matrix-valued kernel $$\mathbb{K}_{\ell,k,m,1,h}(x)=[P_1^T(-D)P_1(D)I_{m\times m}-P_1(D)P_1^T(-D)]\Phi_{\ell,k,1,h}(x),$$   due to the fact that $P_1^T(-D)P_1(D)\Phi_1=\delta$ and $\delta*\psi_{\ell,k,h}=\psi_{\ell,k,h}$.  In particular, when $m=3$, we can also construct a $3\times 3$ matrix-valued kernel $$\mathbb{J}_{\ell,k,3,1,h}(x)=P_1(D)P_1^T(-D)\Phi_{\ell,k,1,h}(x),$$ such that $P_1(D)\times \mathbb{J}_{\ell,k,3,1,h}=(\bold 0_3^T, \bold 0_3^T, \bold 0_3^T)$. This provides a general framework for constructing matrix-valued kernels with prescribed physical constraints from generalized B-splines $\Phi_{\ell,k,1,h}$, containing the tension splines and the Mat\'ern function of \cite{Fasshauer} as special cases.

We turn next to the multivariate case. As an example, we take $P_1(\nabla)=\bm\lambda+\nabla$ with a vector $\bm\lambda=(\lambda_1,\lambda_2,\cdots, \lambda_d)^T$ of length $\|\bm\lambda\|=\sqrt{\bm\lambda^T\bm\lambda}$. Since the classical divergence-free case $\bm\lambda=\bold 0_d^T$ has already been discussed, we assume $\bm\lambda\neq \bold 0_d^T$. Then $P_1^T(-D)P_1(D)=\|\bm\lambda\|^2-\Delta$ whose Green function (Sobolev spline or Mat$\acute{e}$rn function) takes the form \cite{Fasshauer}
\[
\Phi_1(\bold x) = \frac{1}{(2\pi)^{d/2} \|\bm\lambda\|^{\frac{d-2}{2}}} \|\bold x\|^{-\frac{d-2}{2}} K_{\frac{d-2}{2}}(\|\bm\lambda\| \cdot \|\bold x\|)
\]
with $K_{d/2-1}$ being the modified Bessel function of the second kind of order $d/2-1$. Similarly,  by letting $\Phi_{\ell,k,1,h}=\Phi_1*\psi_{\ell,k,h}$, we can construct two matrix-valued kernels $$\mathbb{K}_{\ell,k,m,1,h}(\bold x)=[(\|\bm\lambda\|^2-\Delta)I_{m\times m}-(\bm\lambda+\nabla)(\bm\lambda-\nabla)^T]\Phi_{\ell,k,1,h}(\bold x)$$   such that $(\bm\lambda-\nabla)\cdot\mathbb{K}_{\ell,k,m,1,h}=\bold 0_m$. In particular, when $m=3$, we can also construct a $3\times 3$ matrix-valued kernel
$$\mathbb{J}_{\ell,k,3,1,h}(\bold x)=(\bm\lambda+\nabla)(\bm\lambda-\nabla)^T\Phi_{\ell,k,1,h}(\bold x),$$ such that $(\bm\lambda+\nabla)\times \mathbb{J}_{\ell,k,3,1,h}=(\bold 0_3^T, \bold 0_3^T, \bold 0_3^T)$.
We stress that this last kernel is available even though  the hypothesis $P_1(-D)=(-1)^{r_1}P_1(D)$ used for Identity \eqref{generalidentity} is violated. Indeed, every column of $P_1(D)P_1^T(-D)$ is a scalar multiple of $P_1(D)$, and $P_1(D)\times P_1(D)=\bold 0_3^T$ for any $P_1$. We note that the  hypothesis $P_1(-D)=(-1)^{r_1}P_1(D)$ is needed only for Identity \eqref{generalidentity} and hence only for the error estimates of Lemma \ref{crosprojectionconvolutionerror} and Theorem \ref{errorestimateforderi2}, not for the structure preservation itself.
More importantly, the dot-product constraint $\nabla\cdot \bold u_m=\bm\lambda \cdot \bold u_m$ extends its classical divergence-free counterpart and enables an efficient and accurate identification of sources and sinks within vector fields. Precisely, if $\bm\lambda\cdot\bold u_m$ changes signs in the domain, that is, the corresponding vector field possesses both sources and sinks, then the constraint expresses the source density  as an \emph{algebraic} ($\bm\lambda \cdot \bold u_m$) rather than a differential functional ($\nabla\cdot \bold u_m$) of the field. In addition, since $Q_{\mathbb{K}_{\ell,k,m,1,h}}\bold u_m$ satisfies the same constraint exactly, its source density goes as $\bm\lambda \cdot(Q_{\mathbb{K}_{\ell,k,m,1,h}}\bold u_m)$.  Therefore, we can use  $\bm\lambda \cdot(Q_{\mathbb{K}_{\ell,k,m,1,h}}\bold u_m)$ instead of $\nabla\cdot (Q_{\mathbb{K}_{\ell,k,m,1,h}}\bold u_m)$ to locate  sources and sinks of $\bold u_m$  at the full order $2k$ of the scheme rather than at the order $2k-1$ that differentiation would cost. Subsection \ref{numericalsimulation} carries this out numerically.
\subsection{\textit{The case $n\geq 2$}}
In this case the physical structure of a vector-valued function is described by a system of linear differential operators with constant coefficients. For $d=1$ one obtains a system of ordinary differential equations whose Green function is available from the technique of the previous case ($n=1$, $d=1$), we therefore concentrate on $d\geq 2$.

 We begin with first-order constraints. The first example is the steady-state acoustics constraints \cite{Atteia1}: $\partial_{x_1}p=0$, $\partial_{x_2}p=0$, $\partial_{x_1}u_1+\partial_{x_2}u_2=0$, where $\bold u_3=(u_1,u_2,p)^T$ with $p$ the pressure and $(u_1,u_2)$ the velocity. Setting $P_1(D)=(0,0, \partial_{x_1})^T$, $P_2(D)=(0,0, \partial_{x_2})^T$, $P_3(D)=(\partial_{x_1},\partial_{x_2},0)^T$, these constraints take the form $P_l(-D)\cdot \bold u_3=0$ for $l=1,2,3$, which corresponds to $d=2$, $m=n=3$, and $r_1=r_2=r_3=1$. Note that $P_1(-i\bm\omega)$ and $P_2(-i\bm\omega)$ are parallel for every $\bm\omega$, so that $G(-i\bm\omega)$ is singular at every frequency; this is precisely the situation for which the pseudo-inverse in \eqref{proop} is needed. A direct computation yields
 \begin{equation*}
\mathbb{P}(\bm \omega) = \begin{pmatrix}\frac{\omega_2^2}{\|\bm\omega\|^2} & \frac{-\omega_1\omega_2}{\|\bm\omega\|^2} & 0\\[6pt] \frac{-\omega_1\omega_2}{\|\bm\omega\|^2} & \frac{\omega_1^2}{\|\bm\omega\|^2} & 0\\[4pt] 0&0&0\end{pmatrix},\quad \Phi_3 =-\phi_1I_{3\times 3},
 \end{equation*}
with $\phi_1$ being the two-dimensional thin plate spline  ($\Delta\phi_1 = \delta$).
 Consequently,  we have \begin{equation*}
\mathbb{K}_{\ell,k,3,3,h}(\bold x) = \begin{pmatrix}(\Delta I_{2\times 2}-\nabla \nabla^T)(\phi_1*\psi_{\ell,k,h})(\bold x) & 0_{2\times1}\\ 0_{1\times2} & 0\end{pmatrix},
\end{equation*}
where we used $\Delta(\phi_1*\psi_{\ell,k,h}) = (\Delta\phi_1)*\psi_{\ell,k,h} = \delta*\psi_{\ell,k,h} =\psi_{\ell,k,h}$ to rewrite $I_{2\times 2}\psi_{\ell,k,h} - \nabla\nabla^T(\phi_1*\psi_{\ell,k,h})$ as $(\Delta I_{2\times 2}-\nabla\nabla^T)(\phi_1*\psi_{\ell,k,h})$. Thus the kernel projects the velocity onto its divergence-free part and annihilates the pressure, in accordance with the constraints: only divergence-free velocities with vanishing pressure gradient survive.

We consider next the coupled divergence-free fields of magnetohydrodynamics.  Let $\bold u=(u_1,u_2,u_3)^T$ and $\bold B=(B_1,B_2,B_3)^T$ be corresponding velocity and magnetic field in magnetohydrodynamics \cite{Li}. We set $\bold u_6=(\bold u^T, \bold B^T)^T$ and $P_1(D)=(\nabla^T,\bold 0_3)^T$, $P_2(D)=(\bold 0_3,\nabla^T)^T$. Then the coupled divergence-free constraints read
$P_1(-D)\cdot\bold u_6=\nabla\cdot \bold u=0$ and $P_2(-D)\cdot\bold u_6=\nabla\cdot \bold B=0$, which correspond to $d=3,m=6$, $n=2$, $r_1=r_2=1$. Moreover, we have
 \begin{equation*}
\mathbb{P}(\bm \omega) = \begin{pmatrix}I_{3\times 3} - \frac{\bm\omega\bm\omega^T}{\|\bm\omega\|^2} & 0_{3\times3}\\[2pt] 0_{3\times3} & I_{3\times3} - \frac{\bm\omega\bm\omega^T}{\|\bm\omega\|^2}\end{pmatrix},\quad \Phi_2 = -\phi_1I_{2\times 2},
 \end{equation*} and  \begin{equation*}
\mathbb{K}_{\ell,k,6,2,h}(\bold x) = \begin{pmatrix}(\Delta I_{3\times 3}-\nabla\nabla^T)(\phi_1*\psi_{\ell,k,h})(\bold x) & 0_{3\times3}\\[2pt] 0_{3\times3} & (\Delta I_{3\times 3}-\nabla\nabla^T)(\phi_1*\psi_{\ell,k,h})(\bold x)\end{pmatrix},
\end{equation*}
which is exactly two copies of the divergence-free kernel of \cite{Gaoetal4} along the block diagonal.

 Similarly, consider the divergence-curl case \cite{ChenandSuter}. Let $\bold u_6=(\bold u^T, \bold v^T)^T$ with $\bold u=(u_1,u_2,u_3)^T$ the velocity and $\bold v=(v_1,v_2,v_3)^T$ the vorticity. The constraints $\nabla\cdot \bold u=0$ and $\nabla\times \bold v=\bold 0_{3}^T$ can be recast as $P_l(-D)\cdot \bold u_6=0$ for $1\leq l\leq 4$ with $P_1(D)=(\partial_{x_1},\partial_{x_2},\partial_{x_3},0,0,0)^T$, $P_2(D)=(0,0,0,0, -\partial_{x_3},\partial_{x_2})^T$, $P_3(D)=(0,0,0, \partial_{x_3},0,-\partial_{x_1})^T$, $P_4(D)=(0,0,0, -\partial_{x_2},\partial_{x_1},0)^T$, corresponding to $d=3$, $m=6$, $n=4$, and $r_1=r_2=r_3=r_4=1$. Here the three curl symbols obey the pointwise relation $\omega_1P_2(-i\bm\omega)+\omega_2P_3(-i\bm\omega)+\omega_3P_4(-i\bm\omega)=\bold 0_6^T$, the frequency-domain expression of $\nabla\cdot(\nabla\times\,\cdot\,)\equiv 0$, so that $G(-i\bm\omega)$ is again singular for every $\bm\omega$. A direct computation gives
 \begin{equation*}
\mathbb{P}(\bm \omega) = \begin{pmatrix}I_{3\times 3} - \frac{\bm\omega\bm\omega^T}{\|\bm\omega\|^2} & 0_{3\times3}\\[4pt] 0_{3\times3} & \frac{\bm\omega\bm\omega^T}{\|\bm\omega\|^2}\end{pmatrix},\quad \Phi_4 =-\phi_1I_{4\times 4},
 \end{equation*} and  \begin{equation*}
\mathbb{K}_{\ell,k,6,4,h}(\bold x) = \begin{pmatrix}(\Delta I_{3\times 3}-\nabla\nabla^T)(\phi_1*\psi_{\ell,k,h})(\bold x) & 0_{3\times3}\\[2pt] 0_{3\times3} & \nabla\nabla^T(\phi_1*\psi_{\ell,k,h})(\bold x)\end{pmatrix}.
\end{equation*}
  Observe that the kernel decouples: its first three components produce the divergence-free velocity field, and its last three the curl-free vorticity field.

As a toy example, we take $d=m=n=3$ and rewrite the curl-free constraint $\nabla\times \bold u_3=\bold 0_{3}^T$ as three dot-product constraints
$P_l(-D)\cdot \bold u_3=0$, $l=1,2,3$, with $P_1(D)=(0, -\partial_{x_3}, \partial_{x_2})^T$, $P_2(D)=(\partial_{x_3}, 0, -\partial_{x_1})^T$, and $P_3(D)=(-\partial_{x_2}, \partial_{x_1}, 0)^T$. Then $$\mathbb{P}(\bm \omega)=\frac{\bm\omega\bm\omega^T}{\|\bm\omega\|^2}, \quad \Phi_3 =-\phi_1I_{3\times 3},\quad  \text{and}\quad \mathbb{K}_{\ell,k,3,3,h}(\bold x)=\nabla\nabla^T(\phi_1*\psi_{\ell,k,h})(\bold x),$$ which is exactly the curl-free kernel constructed in reference \cite{Fisher}. Moreover, $\mathbb{K}_{\ell,k,3,3,h}$ coincides with the kernel $\mathbb{J}_{\ell,k,3,1,h}$ associated with the cross-product constraint \eqref{otherphysicalinvariants} for $P_1(D)=\nabla$. This illustrates the duality between the dot-product constraints \eqref{physicalinvariants} with $n=m=3$ and the single cross-product constraint \eqref{otherphysicalinvariants} generated by $\nabla$.

All examples so far are of first order, and each of them decouples: the constraints act on disjoint blocks of components, so that the resulting kernel is a block assembly of divergence-free and curl-free blocks. We go further with another example (Saint-Venant Constraints) that is neither. It is of \emph{second} order, its constraints \emph{couple} all components, its Gram symbol degenerates, and the resulting kernel is nevertheless available in closed form in terms of the two lowest thin plate splines $\phi_1$ and $\phi_2$ alone.

Let $d=3$ and let $\bm\varepsilon$ be a symmetric strain tensor written as
$\bold u_6=(\varepsilon_{11},\varepsilon_{22},\varepsilon_{33},\sqrt2\varepsilon_{23},\sqrt2\varepsilon_{13},\sqrt2\varepsilon_{12})^T$, where the factors $\sqrt2$ (Mandel normalisation) are chosen so that the Euclidean inner product on $\R^6$ agrees with the tensor inner product on symmetric tensors. The Saint-Venant compatibility conditions $\mathrm{inc}(\bm\varepsilon)=0$ of linear elasticity express that $\bm\varepsilon$ is the symmetrised gradient of a displacement field. Taking their three diagonal components gives constraints of the form \eqref{physicalinvariants} with $m=6$, $n=3$ and $r_1=r_2=r_3=2$, namely
\begin{equation*}
\begin{split}
P_1(D)&=(0,\ \partial_{x_3}^2,\ \partial_{x_2}^2,\ -\sqrt2\,\partial_{x_2}\partial_{x_3},\ 0,\ 0)^T,\\
P_2(D)&=(\partial_{x_3}^2,\ 0,\ \partial_{x_1}^2,\ 0,\ -\sqrt2\,\partial_{x_1}\partial_{x_3},\ 0)^T,\\
P_3(D)&=(\partial_{x_2}^2,\ \partial_{x_1}^2,\ 0,\ 0,\ 0,\ -\sqrt2\,\partial_{x_1}\partial_{x_2})^T.
\end{split}
\end{equation*}
A direct computation gives the determinant $\det G(-i\bm\omega)=2\,\omega_1^2\omega_2^2\omega_3^2\|\bm\omega\|^6$, which vanishes on the three coordinate planes. Thus $G$ is \emph{not} invertible on a set that the frequency integrals, and the Moore--Penrose pseudo-inverse in \eqref{proop} is genuinely required, this is the first example in which the degeneracy is not a consequence of a block structure.

Nevertheless, $\mathbb{P}$ admits an explicit representation, as the solution set of $\mathrm{inc}(\bm\varepsilon)=0$ coincides precisely with the range of the symmetrised gradient. Denote by $\Gamma(D)$ the $6\times3$ symmetrised-gradient operator under the above normalisation, so that its symbol is $\Gamma(i\bm\omega)=i\mathcal{A}(\bm\omega)$ with $\mathcal{A}(\bm\omega)=(\mathcal{A}_1(\bm\omega),\mathcal{A}_2(\bm\omega),\mathcal{A}_3(\bm\omega))$ and $\mathcal{A}_1(\bm\omega)=(\omega_1,0,0,0,\omega_3/\sqrt2,\omega_2/\sqrt2)^T$,
$\mathcal{A}_2(\bm\omega)=(0,\omega_2,0,\omega_3/\sqrt2,0,\omega_1/\sqrt2)^T$,
$\mathcal{A}_3(\bm\omega)=(0,0,\omega_3,\omega_2/\sqrt2,\omega_1/\sqrt2,0)^T$.
One verifies $P_l^T(i\bm\omega)\mathcal{A}(\bm\omega)=\bold 0_3^T$ for $l=1,2,3$, so $\mathrm{Ran}\,\mathcal{A}\subseteq \ker P^*$, with $\operatorname{rank}\mathcal{A}=3$ for every $\bm\omega\neq\bold 0$.
Furthermore, whenever $\omega_1\omega_2\omega_3\neq0$, $\operatorname{rank}P^*=3$, so that $\dim\ker P^*=3=\dim\mathrm{Ran}\,\mathcal{A}$ and thus the inclusion is an equality. Thus the two orthogonal projectors onto that common subspace agree, i.e., $\mathbb{P}=\mathcal{A}(\mathcal{A}^T\mathcal{A})^{-1}\mathcal{A}^T$.
On coordinate planes, by contrast, the inclusion is strict and the two projectors differ. For example, at $\bm\omega=(0,\omega_2,\omega_3)$, both  $P_2(-i\bm\omega)$ and $P_3(-i\bm\omega)$  collapse to scalar multiples of the first standard basis vector. In this case, $\operatorname{rank}P=2$, $\ker P^*$ is four-dimensional, and $\mathbb{P}$ has rank four, whereas  $\mathcal{A}(\mathcal{A}^T\mathcal{A})^{-1}\mathcal{A}^T$ remains of rank three. This shows that the two Fourier multipliers coincide away from a Lebesgue-null set; they thus induce the same bounded operator on $L^2$ and define identical kernels in the sense of tempered distributions. Thus we may compute $\mathbb{K}_{\ell,k,6,3,h}$ from the projector $\mathcal{A}(\mathcal{A}^T\mathcal{A})^{-1}\mathcal{A}^T$.

Indeed, the relation $\mathcal{A}^T\mathcal{A}=\tfrac12\big(\|\bm\omega\|^2I_{3\times 3}+\bm\omega\bm\omega^T\big)$ together with   the Sherman--Morrison formula yields
\begin{equation*}
\big(\mathcal{A}^T\mathcal{A}\big)^{-1}=2\Bigg(\frac{I_{3\times3}}{\|\bm\omega\|^2}-\frac{\bm\omega\bm\omega^T}{2\|\bm\omega\|^4}\Bigg),
\end{equation*}
which in turn leads to
\begin{equation*}
\mathbb{P}(\bm\omega)=\frac{2\,\mathcal{A}\mathcal{A}^T}{\|\bm\omega\|^2}-\frac{(\mathcal{B}\bm\omega)(\mathcal{B}\bm\omega)^T}{\|\bm\omega\|^4},\quad \text{with}\quad \mathcal{B}\bm\omega=(\omega_1^2,\omega_2^2,\omega_3^2,\omega_2\omega_3,\omega_1\omega_3,\omega_1\omega_2)^T.
\end{equation*}
Moreover, since $\|\bm\omega\|^{-2}$ and $\|\bm\omega\|^{-4}$ are the Fourier transforms of $-\phi_1$ and $\phi_2$, the matrix-valued  kernel can be explicitly derived in the spatial domain as
\begin{equation}\label{svkernel}
\mathbb{K}_{\ell,k,6,3,h}=-2\Gamma(D)\Gamma^T(-D)\big(\phi_1*\psi_{\ell,k,h}\big)+\Gamma(D)\nabla\nabla^T\Gamma^T(-D)\big(\phi_2*\psi_{\ell,k,h}\big).
\end{equation}
In addition, by the iterated convolution relations $\phi_1*\phi_{\ell}=\phi_{\ell+1}$ and $\phi_2*\phi_{\ell}=\phi_{\ell+2}$, both convolutions reduce to a finite difference of  thin plate splines. Two features deserve emphasis. First, in contrast with the examples of the previous subsection, the matrix-valued kernel \eqref{svkernel} is not a block assembly of previously known kernels: every entry couples all six components of the strain. Second, although the underlying scalar operator $P^T(-D)P(D)$ is of order twelve, \emph{no} Green function beyond $\phi_1$ and $\phi_2$ is needed.
\section{\textit{Numerical experiments}}
This section consists of two subsections. The first one reports numerical evidence for the two claims that carry the paper: that the predicted convergence orders of Theorem \ref{errorestimateforderi} are attained, and that the physical constraints are preserved \textbf{exactly} rather than approximately. As the second part, we demonstrate a genuinely generalized Helmholtz--Hodge decomposition with our quasi-interpolation.
\subsection{\textit{Numerical simulations}}\label{numericalsimulation}
As an example, we take $d=m=2$, $n=1$, $P_1(\nabla)=\bm\lambda+\nabla$ with $\bm\lambda\neq\bold 0_2^T$ and write $a=\|\bm\lambda\|^2$. The constraint \eqref{physicalinvariants} reads $P_1(-D)\cdot\bold u_2=(\bm\lambda-\nabla)\cdot\bold u_2=0$, that is
\begin{equation}\label{srcsinkconstraint}
\nabla\cdot\bold u_2=\bm\lambda\cdot\bold u_2 .
\end{equation}
Two features motivate us to consider this example. First,  the Gram symbol
$G(-i\bm\omega)=(\bm\lambda-i\bm\omega)^T(\bm\lambda+i\bm\omega)=a+\|\bm\omega\|^2>0$
is invertible at \emph{every} frequency including $\bm\omega=\bold 0$ and thus $\mathbb{P}$ extends smoothly to the origin. Second, and this is the point of the experiment, the constraint \eqref{srcsinkconstraint} converts the \emph{differential} quantity $\nabla\cdot\bold u_2$ into the \emph{algebraic} quantity $\bm\lambda\cdot\bold u_2$, which enables us to locate sources and sinks of $\bold u_2$ more accurately and efficiently if $\bold u_2$ changes signs in the domain.

Since $P_1^T(-D)P_1(D)=a-\Delta$, the corresponding Green function $\Phi_1$ is the Sobolev spline taking the form $\Phi_1(\bold x)=\frac{1}{2\pi}K_0(\|\bm\lambda\|\,\|\bold x\|)$ for $d=2$. Moreover, with $\Theta_{\ell}:=\Phi_1*\phi_{\ell}$ we have $\Phi_{\ell,k,1,h}=(-1)^{\ell}h^{-2\ell}q_{d,\ell,k}(\widetilde{\Delta_h})\Theta_{\ell}$, which in turn leads to a matrix-valued kernel in the spatial domain as $\mathbb{K}_{\ell,k,2,1,h}=(-1)^{\ell}h^{-2\ell}q_{d,\ell,k}(\widetilde{\Delta_h})\big\{[(a-\Delta)I_{2\times2}-(\bm\lambda+\nabla)(\bm\lambda-\nabla)^T]\Theta_{\ell}\big\}$. We emphasize that the function $\Theta_{\ell}$ is elementary: writing $s=\|\bm\omega\|^2$, the partial-fraction identity $\frac{1}{(a+s)s^{2}}=\frac{1}{as^{2}}-\frac{1}{a^{2}s}+\frac{1}{a^{2}(a+s)}$ together with $\widehat{\phi_{\ell}}=(-1)^{\ell}s^{-\ell}$ gives, for $\ell=2$,
\begin{equation*}
\Theta_2=\frac{\phi_2}{a}+\frac{\phi_1}{a^{2}}+\frac{\Phi_1}{a^{2}},
\end{equation*}
a sum of two thin plate splines and one Sobolev spline. In particular, if we chose $\bm\lambda=(1,0)^T$, then the bracket collapses to
\begin{equation*}
(a-\Delta)I_{2\times2}-(\bm\lambda+\nabla)(\bm\lambda-\nabla)^T=
\begin{pmatrix}
-\partial_{x_2}^2 & \partial_{x_2}+\partial_{x_1}\partial_{x_2}\\[2pt]
-\partial_{x_2}+\partial_{x_1}\partial_{x_2} & 1-\partial_{x_1}^2
\end{pmatrix}.
\end{equation*}

  To construct a field satisfying constraint \eqref{srcsinkconstraint}, we begin with a   divergence-free field $\bold w=(\partial_{x_2}g,-\partial_{x_1}g)^T$, where $g(\bold x)=e^{-\|\bold x\|^2}\sin(2x_2)$. The resulting field
$\bold u_2(\bold x)=e^{\bm\lambda\cdot\bold x}\,\bold w(\bold x)$
has genuine sources and sinks inside the computational domain.
Moreover, one can verify that the source density
$\nabla\cdot\bold u_2(\bold x)=2\big(\cos 2x_2-x_2\sin 2x_2\big)e^{-x_1^2+x_1-x_2^2}$
 changes sign across the four interfaces $x_2=\pm0.5384369932$, $\pm1.8217985837$, which are the roots of $\cos 2x_2=x_2\sin 2x_2$; see   Figure \ref{fig:srcsink_exact}. Here the four black lines denote the four exact interfaces. In the sequels, $Q_{\mathbb K}$ and $Q_{\mathbb J}$ are abbreviations of $Q_{\mathbb{K}_{\ell,m,n,h}}$ and $Q_{\mathbb{J}_{\ell,m,n,h}}$ respectively, while $Q_{\ell,h}\bold u_m$ denotes the scalar quasi-interpolant applied to $\bold u_m$ componentwise.

\begin{figure}[t]
    \centering

    \subfigure[Exact field $\bold u_2$, $h=0.05$.]{
        \label{fig:srcsink_exact}
        \includegraphics[height=4.25cm,keepaspectratio]
        {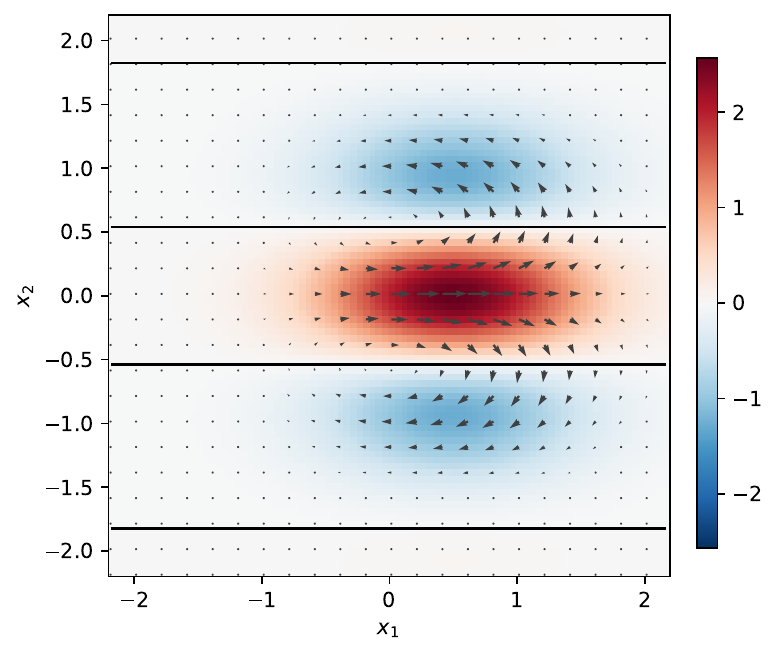}
    }
    \hfill
    \subfigure[$Q_{\mathbb K}\mathbf{u}_2$, $h=0.05$.]{
        \label{fig:srcsink_qk}
        \includegraphics[height=4.25cm,keepaspectratio]
        {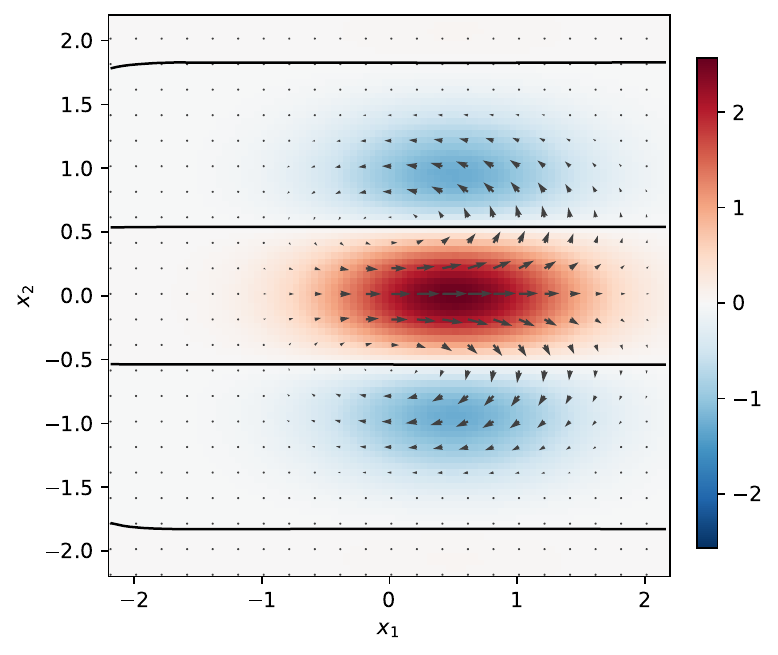}
    }
    \hfill
    \subfigure[Interface reconstruction.]{
        \label{fig:srcsink_interface}
        \includegraphics[height=4.25cm,keepaspectratio]
        {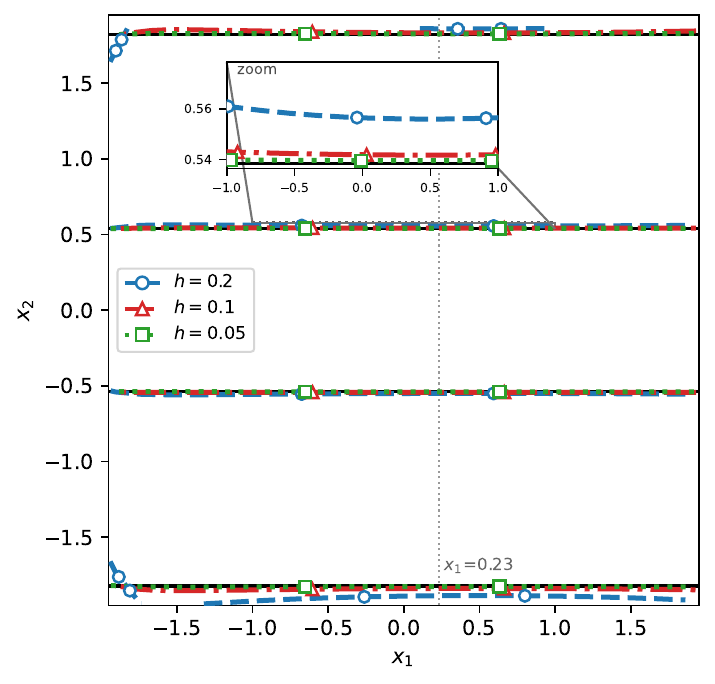}
    }

    \caption{
    Source--sink experiment for $(\boldsymbol{\lambda}-\nabla)\cdot\mathbf{u}_2=0$, $\boldsymbol{\lambda}=(1,0)^T$.
    Panel (a) shows the exact field evaluated on the shifted grid associated with $h=0.05$, and panel (b) shows the corresponding structure-preserving approximation $Q_{\mathbb K}\mathbf{u}_2$.
    In panels (a) and (b), the background color represents the divergence of the corresponding vector field, and the same color scale is used for both panels. The arrows indicate the vector field, while the black zero contour, $\nabla\cdot\mathbf{u}_2=0$, marks the source-sink interface. Panel (c) shows the exact and reconstructed interfaces for $h=0.2$, $0.1$, and $0.05$; the inset magnifies one interface to show the convergence of its reconstructed location.
    }
    \label{fig:srcsink_summary}
\end{figure}

%

 We take $\ell=2$, $k=1$, truncate the lattice sum to $\|\bold jh\|_{\infty}\leq5$ and compute maximum errors over $[-2.2,2.2]^2$. The results are summarized in Table \ref{tab:srcsink} and Figure \ref{fig:srcsinkrates}. Both the approximation error and the source-density error decay at the theoretical order  $2$. These two errors are identical, since for $\bm\lambda=(1,0)^T$ the source density coincides with the first component of the vector field, at which the supremum is achieved. This observed agreement provides numerical evidence for the analytic identity stated above. The constraint residual remains at machine round-off level throughout all tests: its magnitude never exceeds $4\times10^{-11}$,  while the plain componentwise scalar quasi-interpolant $Q_{\ell,h}\bold u_2$ violates constraint \eqref{srcsinkconstraint} by  $10^{-1}$ to $4\times10^{-4}$, i.e., seven to twelve orders of magnitude larger.

\begin{table}[H]
\centering
\caption{Source/sink experiment, $\bm\lambda=(1,0)^T$, $\ell=2$, $k=1$.}\label{tab:srcsink}
\begin{tabular}{c|cc|c|cc}
\hline
$h$ & $\|Q_{\mathbb{K}}\bold u_2-\bold u_2\|_{\infty}$ & order & $\|\bm\lambda\cdot Q_{\mathbb{K}}\bold u_2-\nabla\cdot\bold u_2\|_{\infty}$ & $\|(\bm\lambda-\nabla)\cdot Q_{\mathbb{K}}\bold u_2\|_{\infty}$ & $\|(\bm\lambda-\nabla)\cdot Q_{\ell,h}\bold u_2\|_{\infty}$\\
\hline
$0.4$  & $4.954$e--$01$ & ---    & $4.954$e--$01$ & $7.877$e--$14$ & $1.395$e--$01$\\
$0.2$  & $1.527$e--$01$ & $1.70$ & $1.527$e--$01$ & $6.812$e--$13$ & $2.220$e--$02$\\
$0.1$  & $4.027$e--$02$ & $1.92$ & $4.027$e--$02$ & $5.189$e--$12$ & $3.078$e--$03$\\
$0.05$ & $1.021$e--$02$ & $1.98$ & $1.021$e--$02$ & $3.367$e--$11$ & $4.005$e--$04$\\
\hline
\end{tabular}
\end{table}

\begin{figure}[H]
\centering
\includegraphics[width=\textwidth]{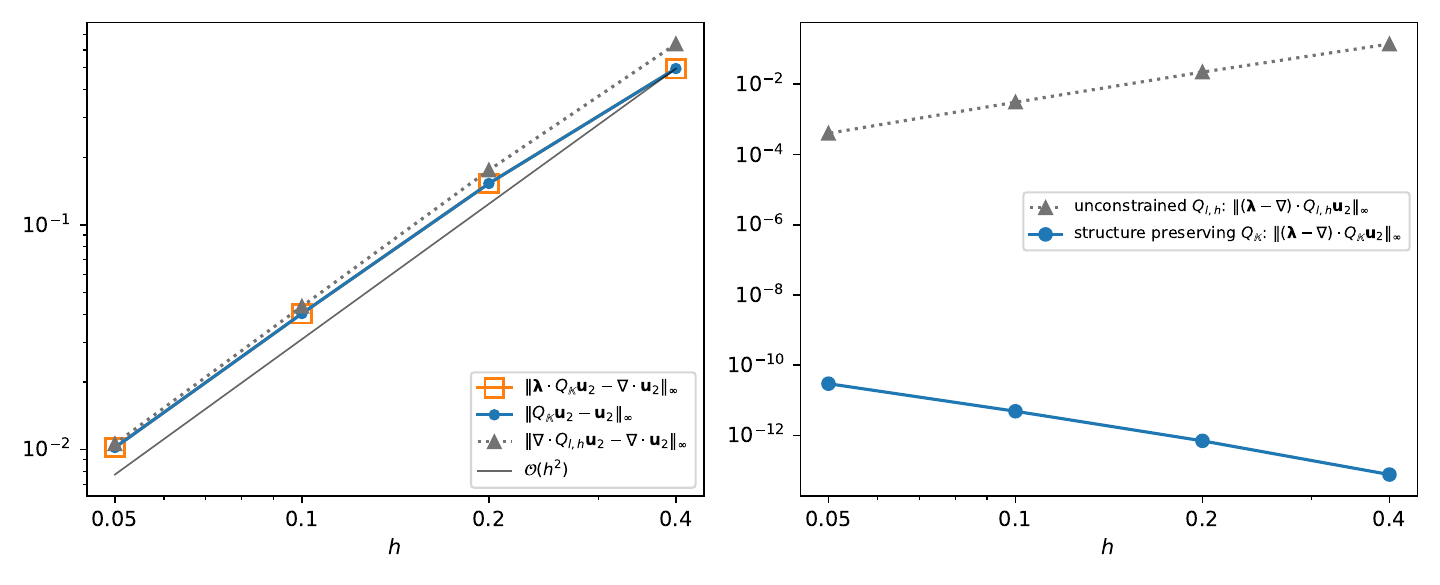}
\caption{Source/sink experiment. Left: convergence of the field and of the source density, against the reference slope $h^{2k}=h^2$. Right: constraint residuals  confirming that the residuals of $Q_{\mathbb{K}}$ is pure round-off.}\label{fig:srcsinkrates}
\end{figure}

Exact structure preservation has a concrete consequence for localization of sources and sinks. Standard source-sink diagnosis requires differentiation of the reconstruction, for which  Theorem \ref{errorestimateforderi} only guarantees an error rate of $\mathcal{O}(h^{2k-1})$. By contrast, since $Q_{\mathbb{K}}\bold u_2$ satisfies constraint \eqref{srcsinkconstraint} \textbf{exactly}, its source density can be evaluated directly via
$\nabla\cdot(Q_{\mathbb{K}_{\ell,k,2,1,h}}\bold u_2)=\bm\lambda\cdot(Q_{\mathbb{K}_{\ell,k,2,1,h}}\bold u_2)$,
without any differentiation. Consequently, the diagnosis inherits the  $\mathcal{O}(h^{2k})$ convergence rate corresponding to $|\alpha|=0$.
  To illustrate the performance of the method, we employ Brent's root-finding algorithm to solve $\bm\lambda\cdot(Q_{\mathbb{K}_{\ell,k,2,1,h}}\bold u_2)(0.23,x_2)=0$ with respect to  $x_2$, thereby capturing all four interfaces. The maximum interface localization errors corresponding to mesh sizes $h=0.4,0.2,0.1,0.05$ are
$1.441\text{e--}01, \ 3.627\text{e--}02,\  9.169\text{e--}03,\ 2.298\text{e--}03$,
 which  yields empirical convergence orders of $1.99$, $1.98$, $2.00$. These numerical results firmly verify that the present interface localization strategy achieves the full order $2k$ with $k=1$.
Figure \ref{fig:srcsink_qk}  demonstrates the reconstructed field under $h=0.05$, while  Figure \ref{fig:srcsink_interface}  compares the numerically reconstructed zero sets with the exact interfaces in two dimensions.
\subsection{\textit{An example of generalized Helmholtz--Hodge decomposition}}\label{sec:genhodge}
We define the differential operator
$$P_1(D)=\big(\partial_{x_1}^2+\partial_{x_2}^2-\partial_{x_3}^2,\ 2\partial_{x_1}\partial_{x_3},\ 2\partial_{x_2}\partial_{x_3}\big)^T=:\big(\mathcal{A}_1,\mathcal{A}_2,\mathcal{A}_3\big)^T,$$
whose Fourier symbol satisfies $P_1(i\bm\omega)=-\bm\sigma(\bm\omega)$, where $\bm\sigma(\bm\omega)=(\omega_1^2+\omega_2^2-\omega_3^2,\,2\omega_1\omega_3,\,2\omega_2\omega_3)^T$ and the norm identity $\|\bm\sigma(\bm\omega)\|^2=\|\bm\omega\|^4$ holds. All assumptions established in Subsection 3.2 are satisfied. Specifically, $P_1(D)$ is a homogeneous differential operator of order $r_1=2$, which yields $P_1(-D)=(-1)^{r_1}P_1(D)=P_1(D)$. Moreover, the relation $\|\bm\sigma(\bm\omega)\|^2=\|\bm\omega\|^4$ immediately implies that $\bm\sigma(\bm\omega)\neq\bold 0_3^T$ for all nonzero $\bm\omega$, which verifies the non-degeneracy condition required in Lemma \ref{ghh}. The corresponding kernel functions can therefore be directly derived.

Observing the identities $\|\bm\sigma\|^2=\|\bm\omega\|^4$ and $P_1^T(-D)P_1(D)=\sum_{i=1}^3\mathcal{A}_i^2=\Delta^2$, we obtain $\Phi_1=\phi_2$ and derive the kernel expressions
\begin{equation}\label{genkernels}
\big(\mathbb{K}_{\ell,k,3,1,h}\big)_{ij}=\big[\Delta^2\delta_{ij}-\mathcal{A}_i\mathcal{A}_j\big]\Theta,
\qquad
\big(\mathbb{J}_{\ell,k,3,1,h}\big)_{ij}=\mathcal{A}_i\mathcal{A}_j\Theta,
\end{equation}
where $\Theta:=\phi_2*\psi_{\ell,k,h}=(-1)^{\ell}h^{-2\ell}q_{3,\ell,k}(\widetilde{\Delta_h})\phi_{\ell+2}$. Both kernels are fourth-order constant-coefficient combinations of a single thin-plate spline. Furthermore, the relation $\Delta^2\phi_{\ell+2}=\phi_{\ell}$ recovers the consistency identity consistent with classical formulations:
\begin{equation}\label{gensum}
\mathbb{K}_{\ell,k,3,1,h}+\mathbb{J}_{\ell,k,3,1,h}=\psi_{\ell,k,h}I_{3\times3}.
\end{equation}
Constrained field decomposition can be similarly established. Since $\mathcal{A}_i\mathcal{A}_j=\mathcal{A}_j\mathcal{A}_i$, the field $P_1(D)\chi$ is annihilated by the curl operator $P_1(D)\times$ for any scalar function $\chi$. Meanwhile, $P_1(D)\times\bold A$ is annihilated by the divergence operator $P_1(-D)\cdot$ for any vector field $\bold A$. These two families of fields are mutually orthogonal at all frequencies, which yields the closed-form field decomposition with generalized Helmholtz--Hodge components:
\begin{equation}\label{genfield}
\bold u_3=P_1(D)\times\bold A+P_1(D)\chi.
\end{equation}

We fix the parameters $d=m=3$, $\ell=2$, and $k=1$, such that the kernels defined in \eqref{genkernels} correspond to fourth-order finite differences of the thin-plate spline $\phi_4=-\|\bold x\|^5/(2880\pi)$. For numerical validation, we adopt the field decomposition in \eqref{genfield} with the scalar component $\chi(\bold x)=e^{-\|\bold x\|^2}$ and the vector component $\bold A(\bold x)=e^{-\|\bold x\|^2}(\sin x_2,\sin x_3,\sin x_1)^T$.  All numerical sampling data are collected on the uniform lattice $h\mathbb{Z}^3$, with spatial indices truncated as $j_\nu=-n_h,\ldots,n_h$ for $\nu=1,2,3$, where the truncation parameter is set to $n_h=\lceil 3/h\rceil$. The corresponding sampling points are defined as $\mathbf{x}_{\mathbf{j}}=h\mathbf{j}=(j_1h,j_2h,j_3h)$. To rigorously quantify the off-grid approximation performance, numerical errors are evaluated at eight fixed off-grid test points:
\[
\begin{aligned}
&(-0.65,-0.65,-0.25),\quad
(-0.25,-0.65,-0.25),\quad
(0.15,-0.65,0.15),\quad
(0.55,-0.65,0.15),\\
&(-0.65,-0.25,0.55),\quad
(-0.25,-0.25,0.55),\quad
(0.15,-0.25,-0.65),\quad
(0.55,-0.25,-0.65).
\end{aligned}
\]
These test points remain fixed for all mesh sizes $h$ and are deliberately positioned away from the sampling lattice, ensuring that the measured errors faithfully reflect the off-grid approximation accuracy of the proposed method. For mesh sizes $h=0.4$, $0.2$, and $0.1$, the total numbers of sampling points are $17^3=4913$, $31^3=29791$, and $61^3=226981$, respectively, while the number of test points is consistently fixed at eight for all cases.

 Numerical results are summarized in Table \ref{tab:genhodge}. Columns 2-5 present the errors of the two field components along with their numerically observed convergence orders, while Column 6 reports the dot-product residual of $Q_{\mathbb{K}}\bold u_3$. The results further validate the exactness of structure preservation of the proposed scheme, as precise residual cancellation strictly relies on the exact equivalence $\sum_i\mathcal{A}_i^2=\Delta^2$ and the accurate relative normalization between $\phi_2$ and $\phi_4$.  Column 7 quantifies the residual errors of the consistency relation in \eqref{gensum}. One implementation point governs this column. The diagonal of $\mathbb{K}$ is evaluated as $\Delta^2\Theta$, literally as \eqref{genkernels} prescribes, that is as fourth-order differences of $\phi_{\ell+2}$, while $\psi_{\ell,k,h}$ is evaluated from $\phi_{\ell}$. The two sides of \eqref{gensum} are therefore computed along independent floating-point paths, and their agreement tests the normalization $\Delta^2\phi_{\ell+2}=\phi_{\ell}$. Had we instead defined $\mathbb{K}:=\psi_{\ell,k,h}I_{3\times3}-\mathbb{J}$, the identity would cancel by construction and the column would verify nothing.

\begin{table}[H]
\centering
\caption{Generalized decomposition with the second-order $P_1$, under $d=m=3$, $\ell=2$, $k=1$.}\label{tab:genhodge}
\begin{tabular}{c|cc|cc|cc}
\hline
$h$ & $\|Q_{\mathbb{K}}\bold u_3-\mathcal{P}^{\cdot}\bold u_3\|_{\infty}$ & order & $\|Q_{\mathbb{J}}\bold u_3-\mathcal{P}^{\times}\bold u_3\|_{\infty}$ & order & $\|P_1(-D)\cdot Q_{\mathbb{K}}\bold u_3\|_{\infty}$ & Identity \eqref{gensum}\\
\hline
$0.4$ & $5.6840$e--$01$ & ---    & $3.6778$e--$01$ & ---    & $2.82$e--$14$ & $1.13$e--$13$\\
$0.2$ & $1.6909$e--$01$ & $1.75$ & $1.0738$e--$01$ & $1.78$ & $1.67$e--$13$ & $8.37$e--$13$\\
$0.1$ & $4.3882$e--$02$ & $1.95$ & $2.7739$e--$02$ & $1.95$ & $1.14$e--$12$ & $3.48$e--$12$\\
\hline
\end{tabular}
\end{table}

Numerical observations demonstrate that both field components converge at the theoretically estimated order $2k=2$. The dot-product constraint is preserved up to machine round-off, and so is the Identity \eqref{gensum}; both residuals grow in proportion to $h^{-2\ell}=h^{-4}$, the amplification factor of the difference representation, rather than in proportion to any power of $h$ that a discretization error would produce. Notably, we do not present the curl-type residual for $Q_{\mathbb{J}}\bold u_3$ in this study. The symmetric property $\mathcal{A}_i\mathcal{A}_j=\mathcal{A}_j\mathcal{A}_i$ inherently enforces $P_1(D)\times\mathbb{J}=\bold 0_3^T$, yielding exact zero residuals in all numerical implementations.

Figures~\ref{fig:genhodge_3d} and~\ref{fig:genhodge_three_slices} show the decomposition realized by the structure-preserving quasi-interpolants. Figure~\ref{fig:genhodge_3d} illustrates the three-dimensional vector field decomposition results for $h=0.1$. Panel (a) plots the exact field $\mathbf{u}_3$, while Panels (b) and (c) display the two reconstructed components $Q_{\mathbb K}\mathbf{u}_3$ and $Q_{\mathbb J}\mathbf{u}_3$, respectively. Their superposition, presented in Panel (d), obeys the identity
$Q_{\mathbb K}\mathbf{u}_3+Q_{\mathbb J}\mathbf{u}_3=Q_{\ell,h}\mathbf{u}_3$,
and precisely recovers the spatial structure of the original vector field.
To further inspect the decomposition performance on typical cross-sections, Figure~\ref{fig:genhodge_three_slices} visualizes field distributions on three representative planes $x_3=-0.75$, $x_3=0$, and $x_3=0.75$. Each row corresponds to one fixed slice position, and the columns sequentially present the exact field $\mathbf{u}_3$, the reconstructed component $Q_{\mathbb K}\mathbf{u}_3$, the complementary component $Q_{\mathbb J}\mathbf{u}_3$, and their superposition $Q_{\mathbb K}\mathbf{u}_3+Q_{\mathbb J}\mathbf{u}_3$. In each subplot, the background color encodes the magnitude of the three-dimensional vector field, while the arrows depict its projection onto the $x_1$-$x_2$ plane. Across all cross-sections, the reconstructed sum agrees with the exact field to the order reported in Table \ref{tab:genhodge}, while the two decomposed components exhibit distinct and physically consistent spatial patterns. Collectively, both three-dimensional global and cross-sectional local visualizations demonstrate that the proposed quasi-interpolants can accurately implement the generalized Helmholtz--Hodge decomposition, while \textbf{exactly} preserving the inherent differential structures of corresponding components.
\begin{figure}[H]
    \centering

    \subfigure[$\mathbf{u}_3$.]{
        \label{fig:genhodge_3d_exact}
        \includegraphics[width=0.4\textwidth]
        {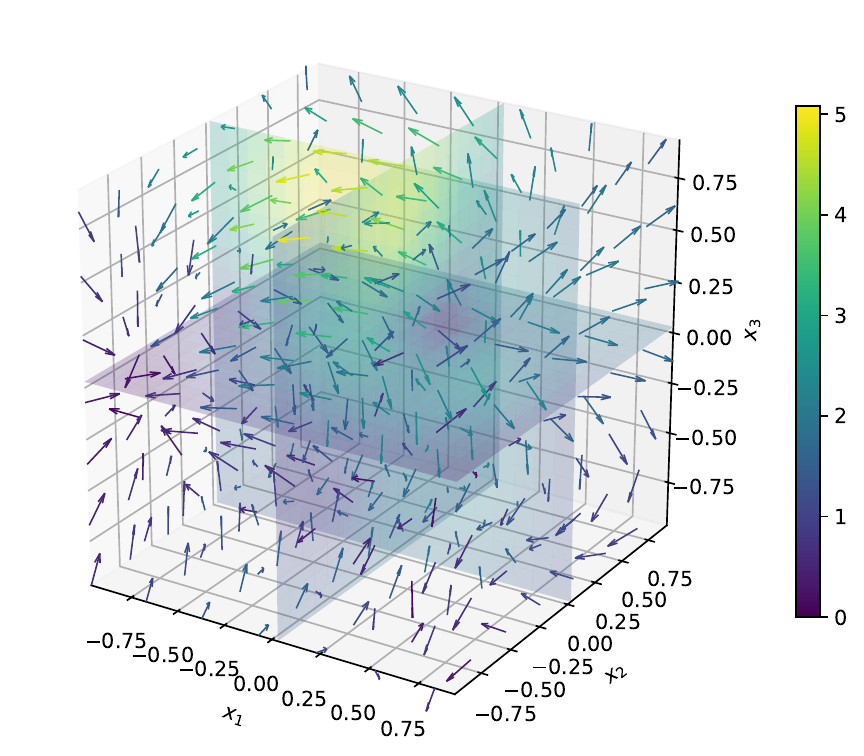}
    }
    \hspace{0.025\textwidth}
    \subfigure[$Q_{\mathbb K}\mathbf{u}_3$.]{
        \label{fig:genhodge_3d_qk}
        \includegraphics[width=0.4\textwidth]
        {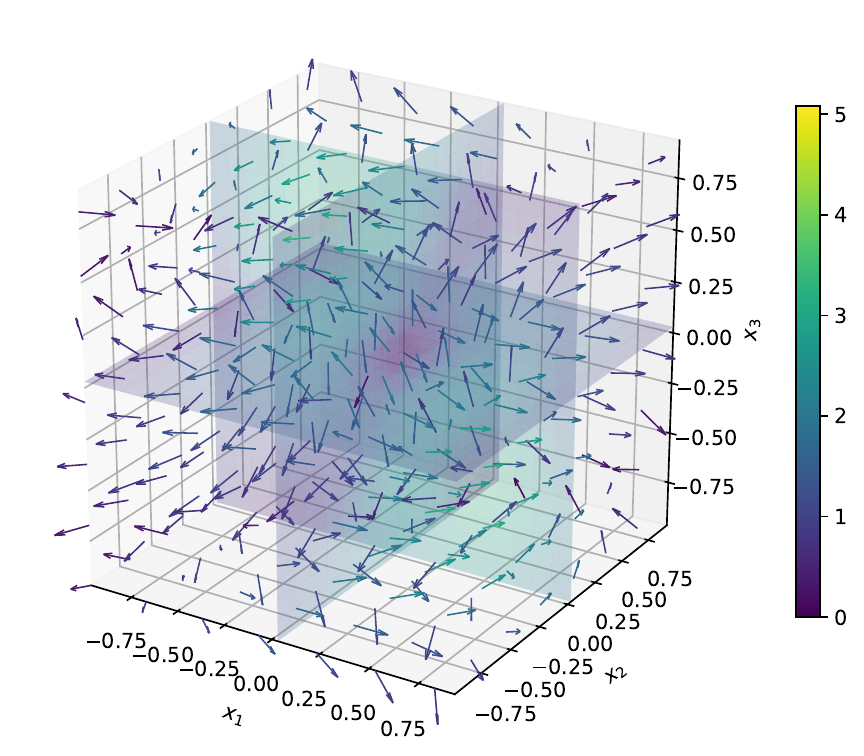}
    }

    \vspace{0.2em}
    \subfigure[$Q_{\mathbb J}\mathbf{u}_3$.]{
        \label{fig:genhodge_3d_qj}
        \includegraphics[width=0.4\textwidth]
        {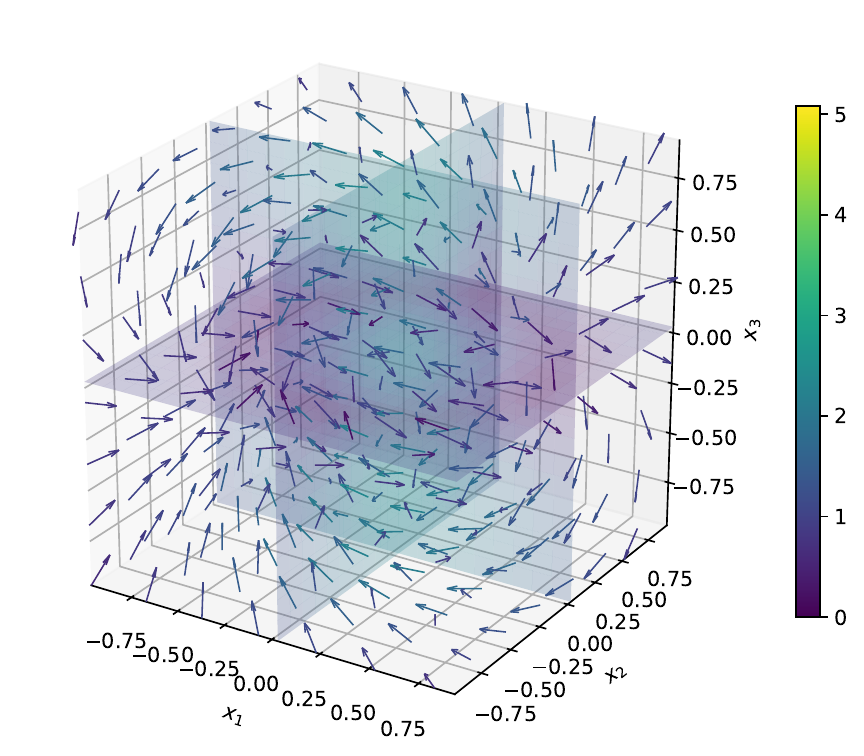}
    }
    \hspace{0.025\textwidth}
    \subfigure[$Q_{\mathbb K}\mathbf{u}_3+
        Q_{\mathbb J}\mathbf{u}_3$.]{
        \label{fig:genhodge_3d_sum}
        \includegraphics[width=0.4\textwidth]
        {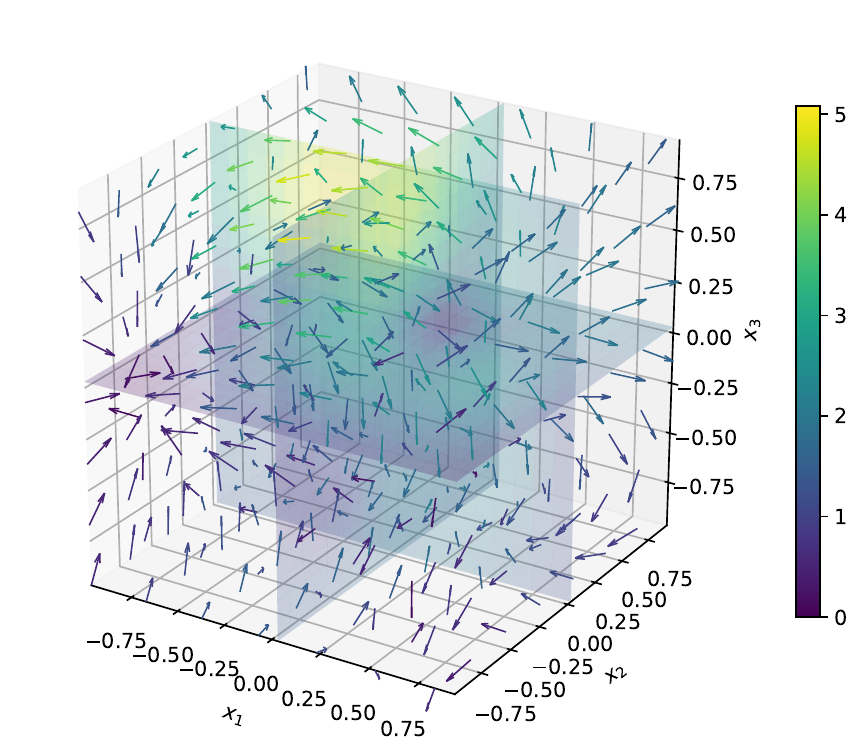}
    }

    \caption{
    Three-dimensional visualization of the generalized Helmholtz--Hodge decomposition for $h=0.1$, with $Q_{\mathbb K}\mathbf{u}_3+Q_{\mathbb J}\mathbf{u}_3 =Q_{\ell,h}\mathbf{u}_3$. The arrows indicate the direction of the corresponding vector field, while the color represents its Euclidean magnitude. A common color scale is used for all four panels, allowing a direct comparison of the field magnitudes.
    }
    \label{fig:genhodge_3d}
\end{figure}

\begin{figure}[H]
    \centering
    \setlength{\tabcolsep}{1pt}
    \renewcommand{\arraystretch}{1.0}

    \begin{tabular}{@{}cccc@{}}

        \includegraphics[width=0.245\textwidth]
        {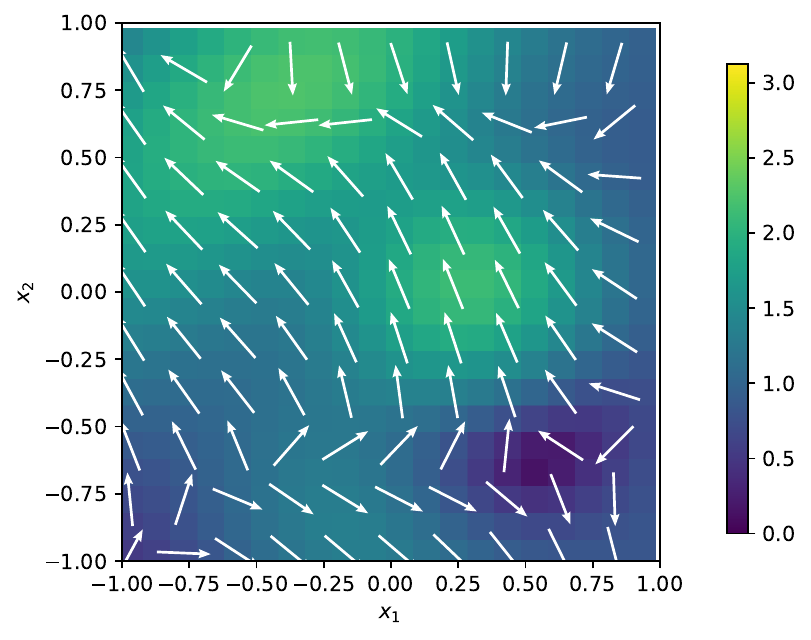}
        &
        \includegraphics[width=0.245\textwidth]
        {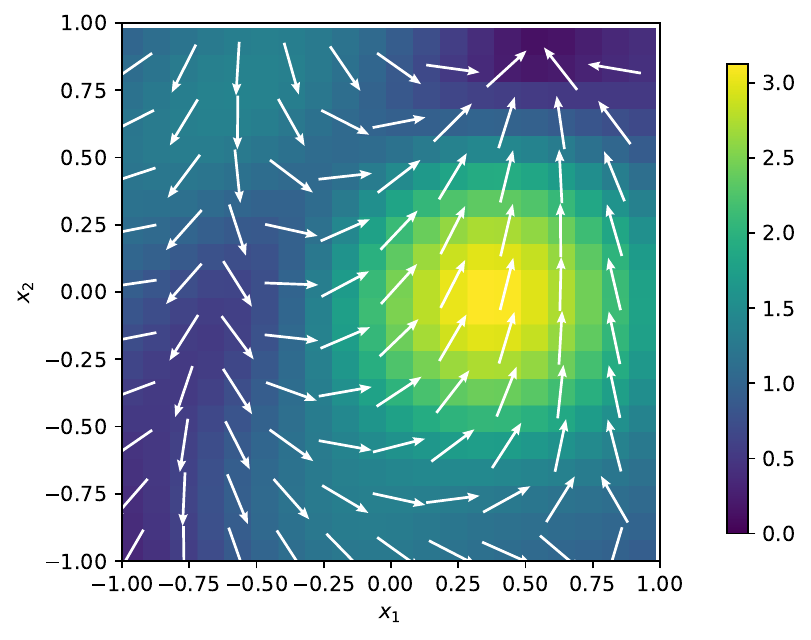}
        &
        \includegraphics[width=0.245\textwidth]
        {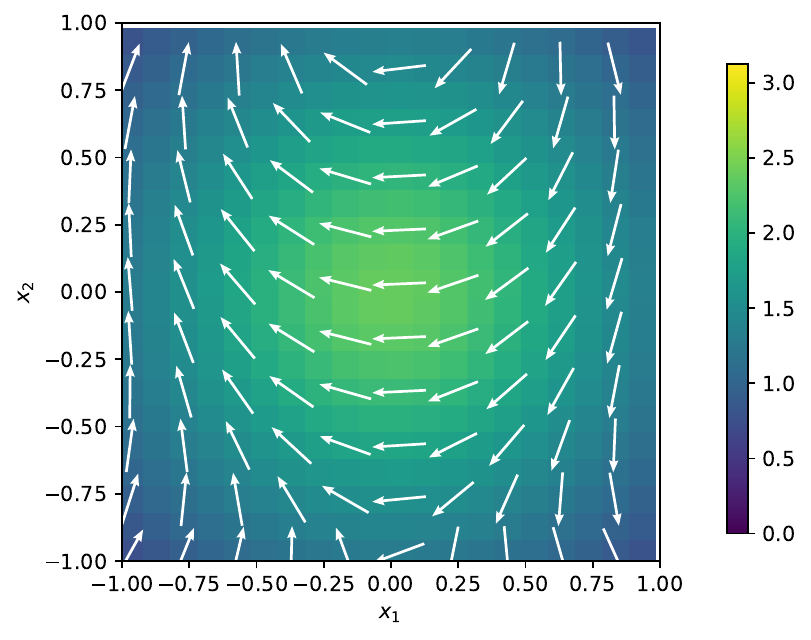}
        &
        \includegraphics[width=0.245\textwidth]
        {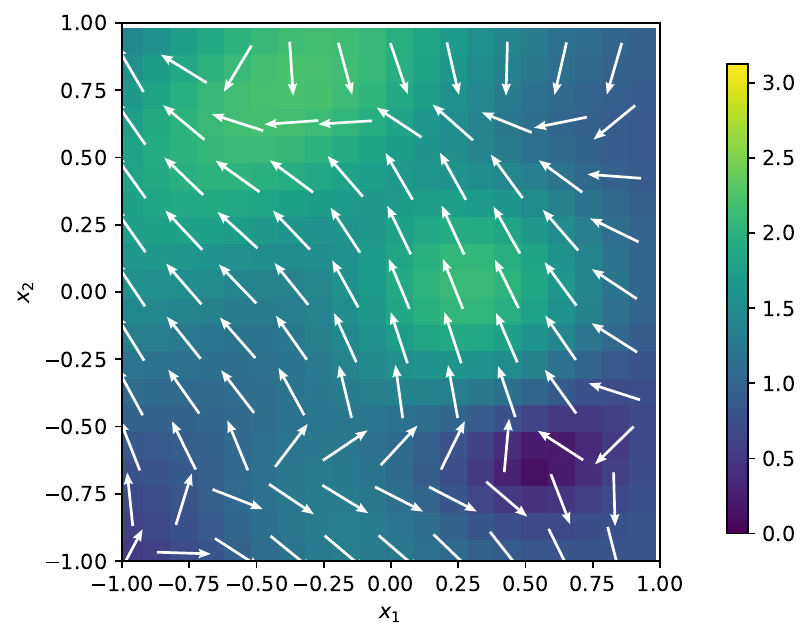}
        \\[-0.15em]

        \multicolumn{4}{c}{
            \small  $x_3=-0.75$
        }
        \\[0.35em]

        \includegraphics[width=0.245\textwidth]
        {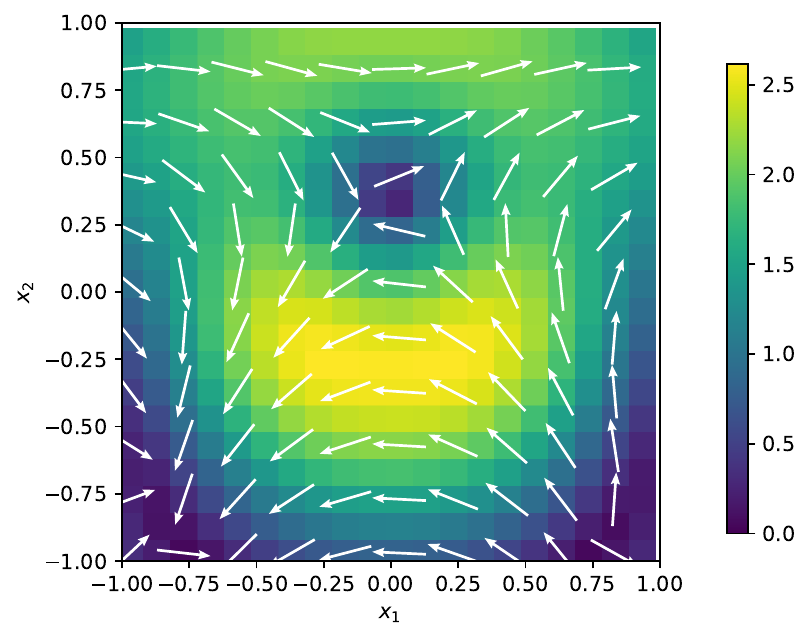}
        &
        \includegraphics[width=0.245\textwidth]
        {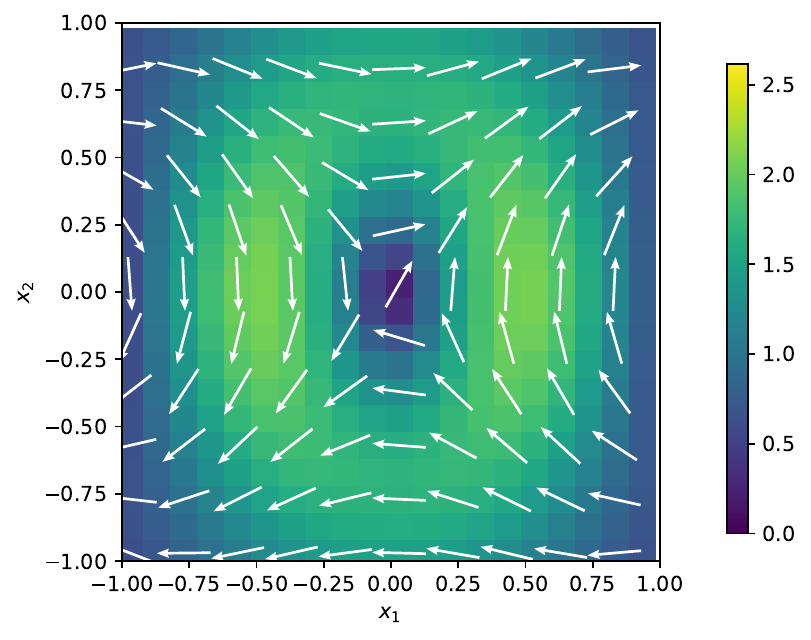}
        &
        \includegraphics[width=0.245\textwidth]
        {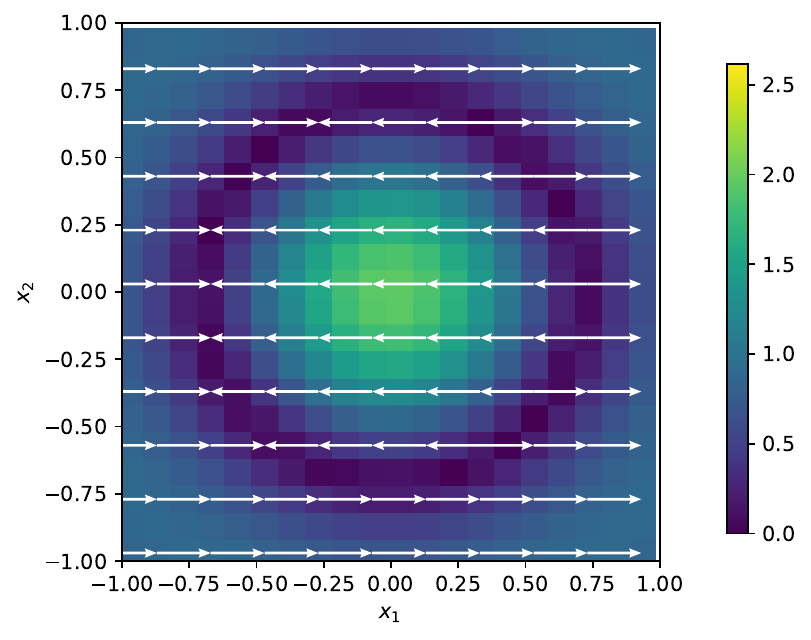}
        &
        \includegraphics[width=0.245\textwidth]
        {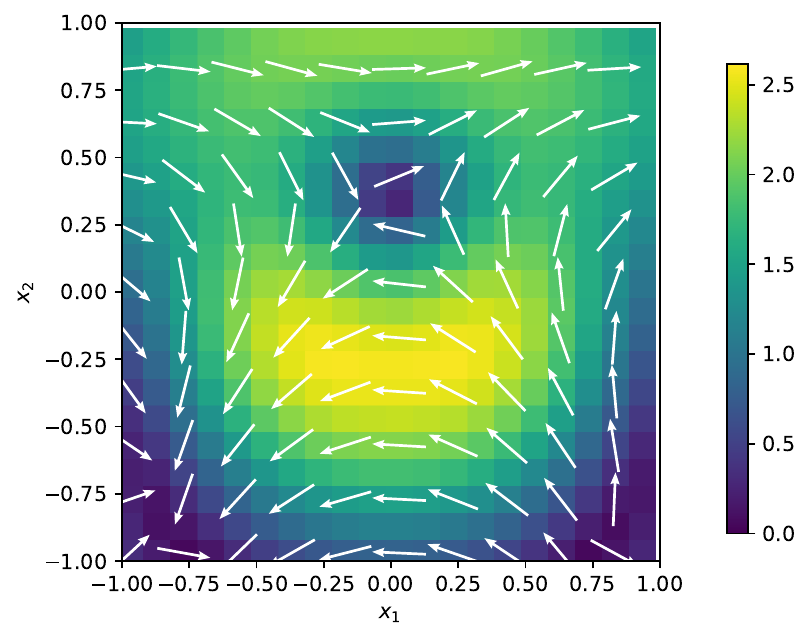}
        \\[-0.15em]

        \multicolumn{4}{c}{
            \small  $x_3=0$
        }
        \\[0.35em]

        \includegraphics[width=0.245\textwidth]
        {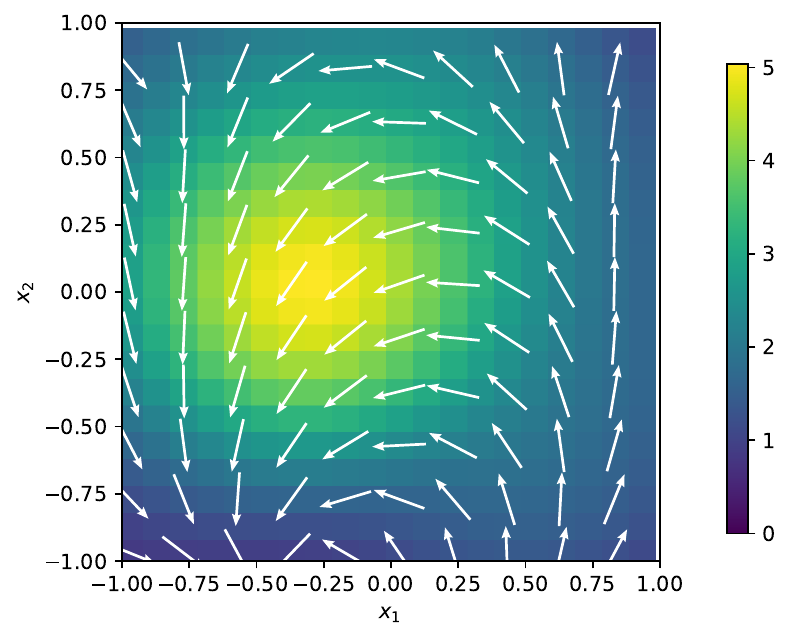}
        &
        \includegraphics[width=0.245\textwidth]
        {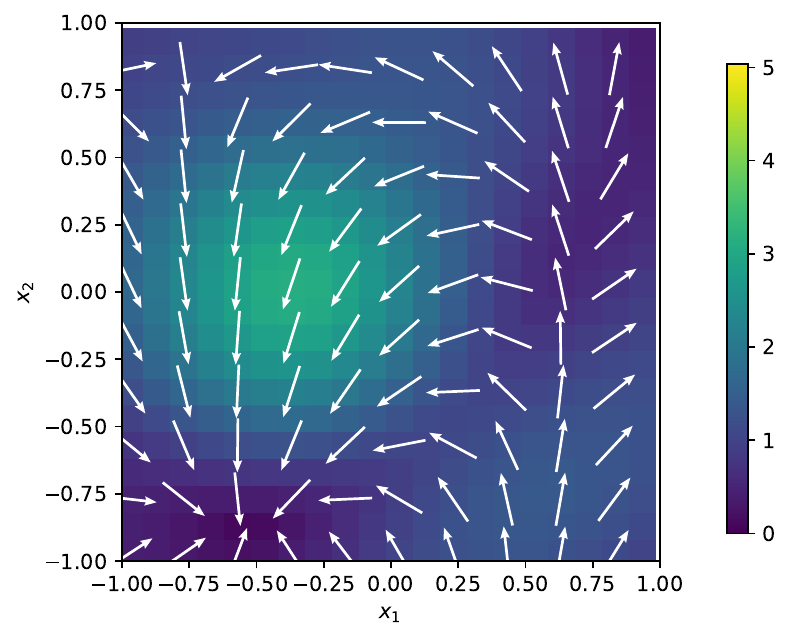}
        &
        \includegraphics[width=0.245\textwidth]
        {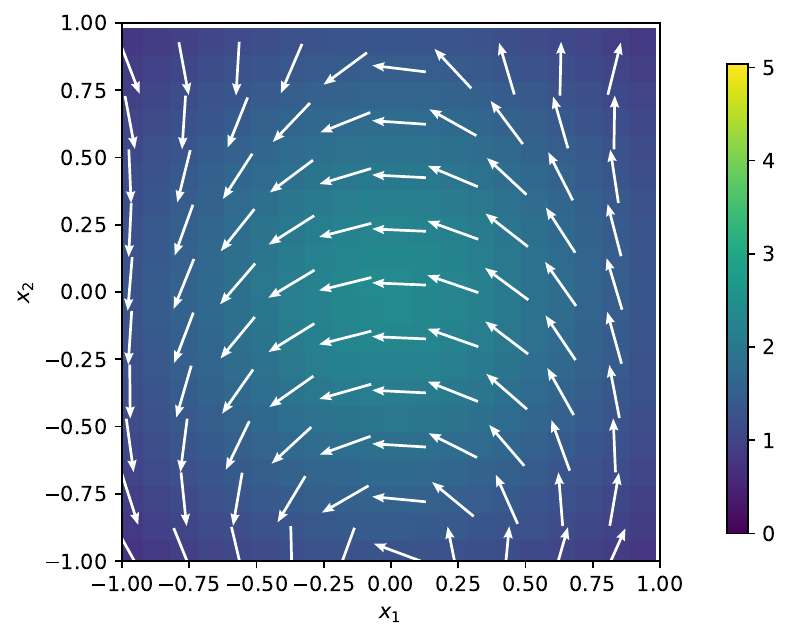}
        &
        \includegraphics[width=0.245\textwidth]
        {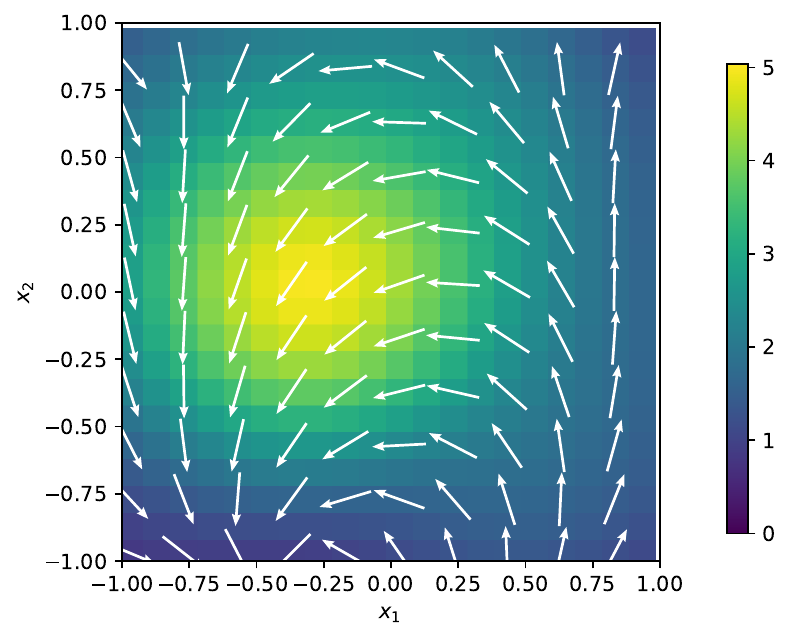}
        \\[-0.15em]

        \multicolumn{4}{c}{
            \small  $x_3=0.75$
        }
        \\[0.55em]

        \small\textbf{(a) $\mathbf{u}_3$.}
        &
        \small \textbf{(b)} $Q_{\mathbb K}\mathbf{u}_3$.
        &
        \small \textbf{(c)} $Q_{\mathbb J}\mathbf{u}_3$.
        &
        \small \textbf{(d)} \textbf{$Q_{\mathbb K}\mathbf{u}_3+Q_{\mathbb J}\mathbf{u}_3$.}

    \end{tabular}

    \caption{
    Cross-sectional visualization of the generalized Helmholtz--Hodge decomposition for $h=0.1$. Rows (a)--(c) correspond to the prescribed planes $x_3=-0.75$, $x_3=0$, and $x_3=0.75$, respectively. From left to right, the columns show the exact field
    $\mathbf{u}_3$, the structure-preserving components $Q_{\mathbb K}\mathbf{u}_3$ and $Q_{\mathbb J}\mathbf{u}_3$, and their reconstructed sum $Q_{\mathbb K}\mathbf{u}_3+Q_{\mathbb J}\mathbf{u}_3$. The background color represents the Euclidean magnitude of the corresponding three-dimensional vector field on each slice, while the arrows show its projection onto the $x_1$--$x_2$ plane.
    }
    \label{fig:genhodge_three_slices}
\end{figure}
 \section{\textit{Conclusions and discussions}}
This paper provides a general framework for the construction of structure-preserving quasi-interpolation schemes for vector-valued functions with intrinsic physical structures  characterized by general dot-product constraints. We propose an approach for constructing matrix-valued kernels that \textbf{analytically} satisfy these constraints column-wise. The resulting quasi-interpolant is a weighted average of discrete function values and is therefore simple to implement. We adopt the scheme to numerical decomposition of  a smooth vector field. Numerical examples in  Subsection $5.2$ validate that both decomposed components can be recovered with theoretically predicted convergence orders.

Numerical simulations in Subsection $5.1$ further reveal a distinctive advantage of exact structure preservation that cannot be explicitly reflected by standard error estimates. For differential constraints that convert a field-based differential functional into an algebraic relation, such as $\nabla\cdot\bold u_m=\bm\lambda\cdot\bold u_m$ for source density calculation,  the exactness structure preservation property of our quasi-interpolant ensure that we can evaluates the targeted functional directly without numerical differentiation. This yields an optimal convergence order of $2k$ instead of $2k-1$ induced by conventional differentiation-based methods. As a typical illustration, we accurately identify the sources and sinks of vector fields governed by the first-order dot-product constraint $(\bm\lambda-\nabla)\cdot\bold u_m=0$. Since the constraint reformulates the source density $\nabla\cdot\bold u_m$ in an algebraic manner $\bm\lambda\cdot\bold u_m$, the reconstruction avoids numerical differentiation and captures field interfaces at the full theoretical convergence order of the scheme.

 Several   research directions remain open for future investigation. A rigorous extension of the current theory to bounded domains will further enhance the practical applicability of the proposed method to realistic numerical simulations. Future work can also explore the performance of the developed schemes for noisy data, where the bias--variance tradeoff of approximation errors dominates the optimal selection of the scale parameter $h$. Finally, the extension to stochastic sampling centers deserves investigation under the framework of probabilistic numerics.
\appendix
\section{\textit{Proof of Identity \eqref{keyrelation}}}\label{app:scaling}
We verify the two identities of \eqref{keyrelation} in turn.

For the first, note that the thin plate spline \eqref{fundamentalsolution} is homogeneous of degree $2\ell-d$ up to a logarithmic term. For odd $d$ this reads $\phi_{\ell}(\bold x)=h^{2\ell-d}\phi_{\ell}(\bold x/h)$. Moreover, substituting $f(\cdot)=g(\cdot/h)$ into the definition of the central difference operator gives
$$(\widetilde{\Delta_h})_j\big[g(\cdot/h)\big](\bold x)=g(\bold x/h-\bold e_j)-2g(\bold x/h)+g(\bold x/h+\bold e_j)=\big[(\widetilde{\Delta_1})_jg\big](\bold x/h),$$
and hence, by induction on the number of difference factors, $q(\widetilde{\Delta_h})\big[g(\cdot/h)\big](\bold x)=\big[q(\widetilde{\Delta_1})g\big](\bold x/h)$ for every polynomial $q$ in the components of the difference operator. Combining the two observations,
$$\psi_{\ell,k,h}(\bold x)=(-1)^{\ell}h^{-2\ell}q_{d,\ell,k}(\widetilde{\Delta_h})\phi_{\ell}(\bold x)=(-1)^{\ell}h^{-2\ell}h^{2\ell-d}\big[q_{d,\ell,k}(\widetilde{\Delta_1})\phi_{\ell}\big](\bold x/h)=h^{-d}\psi_{\ell,k}(\bold x/h),$$
since $-2\ell+(2\ell-d)=-d$.

For even $d$ the same computation produces the additional term $(-1)^{\ell}h^{-d}\ln h\cdot E_{\ell,d}\,q_{d,\ell,k}(\widetilde{\Delta_1})\|\cdot\|^{2\ell-d}(\bold x/h)$, which vanishes. The function $\|\bold y\|^{2\ell-d}=(y_1^2+\cdots+y_d^2)^{\ell-d/2}$ is indeed a polynomial of total degree $2\ell-d$ for even $d$. In addition, a single central difference lowers the degree of a polynomial by two and annihilates polynomials of degree at most one. Moreover, observe that every monomial of $q_{d,\ell,k}$ contains at least $\ell$ difference factors, hence lowers the degree by at least $2\ell>2\ell-d$. Therefore $q_{d,\ell,k}(\widetilde{\Delta_1})\|\cdot\|^{2\ell-d}=0$ and the first identity of \eqref{keyrelation} holds for even $d$ as well.

The second identity follows from the dilation rule of the Fourier transform. Writing $g_h:=h^{-d}g(\cdot/h)$ and substituting $\bold y=\bold x/h$, we have
$$\widehat{g_h}(\bm\omega)=(2\pi)^{-d/2}\int_{\R^d}h^{-d}g(\bold x/h)e^{-i\bm\omega\cdot\bold x}d\bold x=(2\pi)^{-d/2}\int_{\R^d}g(\bold y)e^{-i(h\bm\omega)\cdot\bold y}d\bold y=\widehat{g}(h\bm\omega).$$
Applying this with $g=\psi_{\ell,k}$ gives $\widehat{\psi}_{\ell,k,h}(\bm\omega)=\widehat{\psi_{\ell,k}}(h\bm\omega)$. \qed
\\


\textbf{References}
\begin{thebibliography}{00}



 \bibitem{Amodei} L. Amodei, M. N. Benbourhim, A vector spline quasi-interpolation. In {\sl Wavelets, Images, and Surface Fitting (Chamonix-MontBlanc, 1993)}, A K Peters, Wellesley, MA, 1994, 1--10.
  \bibitem{Atteia} M. Atteia, M. N. Benbourhim, P. G. Casanova, Quasi-interpolant elastic manifolds. In {\sl Reporte de Investigation IIMAS}. UNAM 3 (25) Mexico (1993).
 \bibitem{Atteia1} M. Atteia, Quasi-interpolants and (quasi-) wavelets $P(D)$ manifold, Wavelets, Images, and Surface Fitting, 1994 (29).
 \bibitem{Benbourhim0} M. N. Benbourhim, Div-curl weighted minizing splines, Anal. Appl. 2 (2007) 95--122.
 \bibitem{Benbourhim} M. N. Benbourhim, P. G. Casanova, Generalized variational quasi-interpolants in $(H^m(\Omega))^n$,  Bol. Soc. Mat. Mexicana 3 (1997).
\bibitem{Benbourhim1} M. N. Benbourhim, A. Bouhamidi, Meshless pseudo-polyharmonic divergence-free and curl-free vector fields approximation, SIAM J. Numer. Anal. 42 (2010) 1218--1245.
\bibitem{Bouhamidi} A. Bouhamidi, Weighted thin plate splines, Anal. Appl. 3 (2005) 297--342.
\bibitem{Bouhamidi1} A. Bouhamidi, Pseudo-differential operator associated to the radial basis functions under tension, ESIAM: Proceedings, 20 (2007) 72--82.
\bibitem{Buhmann} M. Buhmann, Convergence of univariate quasi-interpolation using multiquadrics, IMA J. Numer. Anal. 8 (1988) 365--383.
\bibitem{Buhmann1} M. Buhmann, On quasi-interpolation with radial basis functions, J. Approx. Theory  72 (1993) 103--130.
\bibitem{BuhmannandDyn}  M. Buhmann, N. Dyn, D. Levin,   On quasi-interpolation by radial basis functions with scattered centres, Constr. Approx. 11 (1995) 239--254.
\bibitem{BuhmannBook} M. Buhmann, J. J\"ager, Quasi-interpolation (Cambridge Monographs on Applied and Computational Mathematics), Cambridge University Press, 2022.
\bibitem{ChenandSuter} F. Chen, D. Suter, Div-curl vector quasi-interpolation on a finite domain, Math. Comput. Model.  30 (1999) 179--204.
\bibitem{Deriaz} E. Deriaz and V. Perrier, Orthogonal Helmholtz decomposition in arbitrary dimension using divergence-free and curl-free wavelets. Appl. Comput. Harmon. Anal.26 (2009) 249--269.
\bibitem{DoduandRabut0} F. Dodu, C. Rabut, Vectorial interpolation using radial-basis-like functions, Comput. Math. Appl. 43 (2002) 393-411.
\bibitem{DoduandRabut} F. Dodu, C. Rabut, Irrotational or divergence-free interpolation, Numer. Math. 98 (2004) 477--498.
   \bibitem{Duchon} J. Duchon, Splines minimizing rotation-invariant semi-norms in sobolev spaces, In Constructive Theory of
Functions of Several Variables: Proceedings of a Conference Held at Oberwolfach, 85-100, Springer, 1977.
\bibitem{Drake} K. P. Drake, E. J. Fuselier, G. B. Wright, A divergence-free and curl-free radial basis function partition of unity method, SIAM J. Sci. Comput. 43 (2021) A2102--A2126.
\bibitem{EPSTEIN} C. L. Epstein, F. Fryklund, S. D. Jiang, An accurate and efficient scheme for function extension on smooth domains,  SIAM J. Numer. Anal. 63 (2025) 1427--1453.
\bibitem{Fasshauer} G. E. Fasshauer, Q. Ye, Reproducing kernels of generalized Sobolev spaces via a Green function approach with distributional operators, Numer. Math. 119 (2011) 585-611.
\bibitem{Fasshauer1} G. E. Fasshauer, M. J. McCourt, Kernel-based Approximation Methods using Matlab, Interdiscip. Math. Sci. 19, World Scientific, Hackensack, NJ, 2015.
\bibitem{Fisher} N. Fisher, G. E. Fasshauer, W. W. Gao, Quasi-interpolation for the Helmholtz--Hodge decomposition, submitted (arXiv preprint arXiv:2412.04600, 2024).
\bibitem{Wendland7} T. Franz, H. Wendland, Multilevel quasi-interpolation, IMA J. Numer. Anal. 5 (2023) 2934-2964.
\bibitem{Freeden} W. Freeden and T. Gervens, Vector spherical spline interpolation-basic theory and computational aspects, Math. Methods Appl. Sci. 16 (1993) 151--183.
\bibitem{Fuselier} E. Fuselier, Sobolev-type approximation rates for divergence-free and curl-free RBF interpolants, Math. Comput. 77 (2008) 1407--1423.
 \bibitem{Fuselier1} E. Fuselier,  F. Narcowich, J. Ward, G. Wright,  Error and stability estimates for surface-divergence-free RBF interpolants on the sphere, Math. Comput. 78 (2009) 2157--2186.
 \bibitem{Fuselier2}E. Fuselier and G.~Wright, Stability and error estimates for vector field interpolation and decomposition on the sphere with RBF's, SIAM J. Numer. Anal. 47 (2009), 3213--3239.
\bibitem{Gaoetal4} W. W. Gao, G. E. Fasshauer, and N. Fisher, Divergence-free quasi-interpolation, Appl. Comp. Harmon, Anal. 60 (2022) 471--488.
\bibitem{Gaoetal2} W. W. Gao, G. E. Fasshauer, X. P. Sun, X. Zhou, Optimality and regularization properties of quasi-interpolation: deterministic  and stochastic approaches, SIAM J. Numer. Anal. 58 (2020) 2059--2078.
\bibitem{Gaoetal3} W. W. Gao, X. P. Sun, Z. M. Wu, and X. Zhou, Multivariate Monte Carlo approximation based on scattered data, SIAM J. Sci. Comput. 42 (2020) A2262-A2280.
\bibitem{Gaoetal5} W. W. Gao, J. Wang, Z. Sun, G. E. Fasshauer, Quasi-interpolation for high-dimensional function approximation, Numer. Math. 156 (2024) 1855--1885.
\bibitem{Gaoetal1} W. W. Gao, Z. M.  Wu,  Constructing radial kernels with higher-order moment conditions, Adv. Comput. Math.  43 (2017) 1355--1375.
\bibitem{Wendland1} C. Keim, H. Wendland, A high-order, analytically divergence-free approximation method for the time-dependent Stokes problem, SIAM J. Numer. Anal. 54 (2016) 1288--1312.
\bibitem{Kressner} D. Kressner, T. Ni, A. Uschmajew, On the approximation of vector-valued functions by volume sampling, J. Complex. 86 (2025) 101887.
\bibitem{Krieg} D. Krieg, M. Ullrich, Approximation of functions: Optimality sampling and complexity, Acta. Numer. (2006) 273--457.
\bibitem{Leiandjia} J. J. Lei, R. Q. Jia, E. Cheney, Approximation from shift-invariant spaces by integral operators, SIAM J.
Math. Anal., 28 (1997) 481--498.
\bibitem{Li} Y. Li, B. She, On convergence of numerical solutions for the compressible MHD system with weakly divergence-free magnetic field, IMA J.  Numer. Anal. 43(2023) 2169--2197.
 \bibitem{MaandWu} L. M. Ma, Z. M.  Wu, Approximation to the $k$th derivatives by multiquadric quasi-intepolation method, J. Comput. Appl. Math. 2 (2009)
925-932.
\bibitem{MirzaeiandMohammadi} D. Mirzaei, V. Mohammadi, Divergence-free and curl-free moving least squares approximations, J. Comput. Appl. Math. 482 (2026) 117338.
\bibitem{McNally} C. McNally, Divergence-free interpolation of vector fields from point values: exact  $\nabla \cdot  B = 0$  in numerical
simulations. Monthly Notices of the Royal Astronomical Society: Letters. 413 (2011)  L76--L80.
\bibitem{Narcowich} F. Narcowich, J. Ward, Generalized Hermite interpolation via matrix-valued conditionally positive definite functions, Math. Comput. 63 (1994) 661--687.
 \bibitem{Narcowich1} F. Narcowich, J. Ward, G. Wright, Divergence-free RBFs on surfaces,  J. Four. Anal. Appl. 13 (2007) 643--663.
 \bibitem{Ortmann} M. Ortmann, M. Buhmann, High accuracy quasi-interpolation using a new class of generalized multiquadrics, J. Math. Anal. Appl. 538 (2024) 128--359.
 \bibitem{Ortmann1} M. Ortmann, M. Buhmann, On quasi-interpolation and their associated shift-invariant space using a new class of generalized thin plate splines and inverse multiquadrics, J. Fourier Anal. Appl. 31 (2025) 1--35.
 \bibitem{pp} K. Polthier and E. Pruess, Identifying vector field singularities using a discrete Hodge decomposition, Visualization and Mathematics III (H.C. Hedge and K. Polthier eds.) Springer, Berlin.
 \bibitem{CR} C. Rabut, An introduction to Schoenberg's approximation, Comput. Math. Appl. 24 (1991) 139--175.
\bibitem{Rabut0}  C. Rabut, How to build quasi-interpolants: applications to polyharmonic B-splines. In {\sl Curves and Surfaces}, P. J. Laurent, A. Le M\'ehaut\'e, and L. L. Schumaker (eds.). Academic Press, New York, (1991) 391--402.
\bibitem{Rabut}  C. Rabut, Elementary $m$-harmonic cardinal B-splines, Numer. Algor. 2 (1992) 39--62.
  \bibitem{Rabut1}  C. Rabut,  High level $m$-harmonic cardinal B-splines, Numer. Algor. 2 (1992) 63--84.
  \bibitem{Rux} N. Rux, M. Quellmalz, G. Steidl, Slicing of radial fuction: a dimension walk in Fourier space, Samp. Theo. Sign. Proc. Data Anal. 23 (2025) 1-40.
  \bibitem{Ryck}  T. Ryck,  S. Mishra,  Numerical analysis of physics-informed neural networks and related models in physics-informed machine learning. Acta Numer.  33 (2024) 633--713.
   \bibitem{Schabackopti} R. Schaback, Optimal compactly supported functions in Sobolev spaces, Adv. Comput. Math. 52 (2026) https://doi.org/10.1007/s10444-026-10293-9.
  \bibitem{Schabackernel} R. Schaback, Kernel construction techniques,  Eng. Anal. Bound. Elem. 188 (2026) 10678.
 \bibitem{Schabackandwu} R. Schaback, Z. M. Wu, Construction techniques for highly accurate quasi-interpolation operators, J. Approx. Theory 91 (1997) 320--331.
\bibitem{Schoenberg} I. J. Schoenberg, Contributions to the problem of approximation of equidistant data by analytic functions, Quart. Appl. Math. 4 (1946) 45-99 and 112--141.
\bibitem{Wendland3} D. Schr\"ader, H. Wendland, A high-order, analytically divergence-free discretization method for Darcy's problem, Math. Comput. 80 (2011) 263--277.
\bibitem{Schwarz} G. Schwarz,  Hodge decomposition-a method for solving boundary value problems, Lecture Notes in Mathematics, 1607, Springer, Berlin, vii-155.
\bibitem{Wendland6} N. Sharon,  R. S.   Cohen, H. Wendland, On multiscale quasi-interpolation of scattered scalar-and manifold-valued functions, SIAM J. Sci. Comput. 5(2023) A2458-A2482.
\bibitem{Stein} E. Stein, G. Weiss, Introduction to Fourier Analysis on Euclidean Spaces, Princeton Univ. Press, Princeton, 1971.
\bibitem{Strang} G. Strang, G. Fix, A Fourier analysis of the finite-element method, In {\sl Constructive Aspects of Functional Analysis}, G. Geymonat (ed.) (C.I.M.E., Rome, 1973) 793--840.
\bibitem{VenelandBeatson} R. Vennell, R. Beatson, A divergence-free spatial interpolator for large sparse velocity data sets, J. Geophys. Res. 114 (2009) 10--24.
  \bibitem{Wendland5} H. Wendland,  Divergence-free kernel methods for approximating the Stokes problem, SIAM J. Numer. Anal. 47 (2009) 3158--3179.
\bibitem{ZMRS} Z. M. Wu, R. Schaback, Shape preserving properties and convergence of univariate multiquadric quasi-interpolation, Acta. Math. Appl. Sin. 10 (1994) 441--446.
\bibitem{GeneralizedWu} Z. M. Wu, J. P. Liu, Generalized Strang--Fix condition for scattered data quasi-interpolation, Adv. Comput. Math. 23(2005)  201--214.
 \bibitem{PiecewiseWu}  Z. M. Wu, Piecewise functions generated by the solutions of linear ordinary differential equation, Stud. Adv.  Math. 42 (2008) 769--784.
 \end{thebibliography}
\end{document}